\documentclass[11pt,a4paper,oneside]{amsart}

\usepackage{latexsym}
\usepackage{amsmath}
\usepackage{amsthm}
\usepackage{amssymb}
\usepackage{amsfonts}
\usepackage{fancyhdr}
\usepackage{comment}

\usepackage{mathrsfs}                           

\newcommand{\mR}{\mathbf{R}}                    
\newcommand{\mC}{\mathbf{C}}                    
\newcommand{\abs}[1]{\lvert #1 \rvert}          
\newcommand{\norm}[1]{\lVert #1 \rVert}         
\newcommand{\br}[1]{\langle #1 \rangle}         

\newcommand{\mS}{\mathscr{S}}

\newcommand{\mF}{\mathscr{F}}

\newcommand{\ehat}{\,\hat{\rule{0pt}{6pt}}\,}

\newcommand{\re}{\mathrm{Re}}
\newcommand{\im}{\mathrm{Im}}
\newcommand{\supp}{\mathrm{supp}}

\newcommand{\mOp}{\mathrm{Op}}

\theoremstyle{definition}
\newtheorem{thm}{Theorem}[section]
\newtheorem{prop}[thm]{Proposition}

\newtheorem{lemma}[thm]{Lemma}
\newtheorem*{definition}{Definition}

\numberwithin{equation}{section}

\title{Inverse scattering for the magnetic Schr\"odinger operator}
\author{Lassi P\"aiv\"arinta}
\address{Department of Mathematics and Statistics, University of Helsinki}
\email{lassi.paivarinta@helsinki.fi}
\author{Mikko Salo}
\address{Department of Mathematics and Statistics, University of Helsinki}
\email{mikko.salo@helsinki.fi}
\author{Gunther Uhlmann}
\address{Department of Mathematics, University of Washington}
\email{gunther@math.washington.edu}

\date{August 27, 2009}

\begin{document}

\begin{abstract}
We show that fixed energy scattering measurements for the magnetic Schr\"odinger operator uniquely determine the magnetic field and electric potential in dimensions $n \geq 3$. The magnetic potential, its first derivatives, and the electric potential are assumed to be exponentially decaying. This improves an earlier result of Eskin and Ralston \cite{eskinralston} which considered potentials with many derivatives. The proof is close to arguments in inverse boundary problems, and is based on constructing complex geometrical optics solutions to the Schr\"odinger equation via a pseudodifferential conjugation argument.
\end{abstract}

\maketitle

\section{Introduction} \label{sec:intro}

This paper concerns inverse scattering problems at a fixed energy for the magnetic Schr\"odinger operator, defined by 
\begin{equation*}
H = \sum_{j=1}^n (D_j + A_j)^2 + V
\end{equation*}
where $A: \mR^n \to \mR^n$ and $V: \mR^n \to \mR$ are the magnetic and electric potentials, respectively. We will assume that $n \geq 3$ and that the potentials are exponentially decaying. The precise condition will be 
\begin{equation} \label{av_assumptions}
A \in e^{-\gamma_0\br{x}} W^{1,\infty}(\mR^n ; \mR^n),\quad V \in e^{-\gamma_0\br{x}} L^{\infty}(\mR^n ; \mR).
\end{equation}
Here $\gamma_0 > 0$, $\br{x} = (1+\abs{x}^2)^{1/2}$, $D_j = -i \partial/\partial x_j$, and we use the notation $a X = \{ a f ; f \in X \}$ for a positive function $a$ and a function space $X$.

The main result states that if the scattering matrices for two sets of exponentially decaying coefficients coincide at a fixed energy, then the magnetic fields and electric potentials have to be the same. The magnetic field that corresponds to a potential $A$ is given by the $2$-form $dA$, defined by 
\begin{equation*}
dA = \sum_{j,k=1}^n \left( \frac{\partial A_k}{\partial x_j} - \frac{\partial A_j}{\partial x_k} \right) dx^j \wedge dx^k.
\end{equation*}
See Section \ref{sec:preliminaries} for the precise definition of the scattering matrix.

\begin{thm} \label{thm:maintheorem}
Suppose $A$, $V$ and $A'$, $V'$ satisfy \eqref{av_assumptions}, and let $\Sigma_{\lambda}$ and $\Sigma_{\lambda}'$ be the corresponding scattering matrices. If $\Sigma_{\lambda} = \Sigma_{\lambda}'$ for some fixed $\lambda > 0$, then $dA \equiv dA'$ and $V \equiv V'$.
\end{thm}

We now describe earlier results on the problem. In dimension $n\ge 3$ and in the case of no magnetic potential and a compactly supported electric potential (that is, $A \equiv 0$ and $V \in L^{\infty}_{\text{c}}$), uniqueness for the fixed energy scattering problem was proven in \cite{nachman}, \cite{novikov}, \cite{ramm}. In the earlier paper \cite{novikovkhenkin} this was done for small potentials.  For compactly supported potentials, know\-ledge of the scattering amplitude at fixed energy is equivalent to knowing the Dirichlet-to-Neumann map for the Schr\"odinger equation measured on the boundary of a large ball containing the support of the potential (see \cite{uhlmannasterisque} for an account). Then the uniqueness result of Sylvester-Uhlmann \cite{sylvesteruhlmann} for the Dirichlet-to-Neumann map, based on complex geometrical optics solutions, implies uniqueness at a fixed energy for compactly supported potentials. Melrose \cite{melrose} suggested a related proof that uses the density of products of scattering solutions. The fixed energy result was extended by Novikov to the case of exponentially decaying potentials \cite{novikov_exp}. Another proof using arguments similar to the ones used for studying the Dirichlet-to-Neumann map was given in \cite{uhlmannvasy}. We note that exponential decay is a natural assumption given
the counterexamples to uniqueness for inverse scattering at fixed energy for Schwartz potentials due to Grinevich and Novikov
\cite{GrinNov}.

The fixed energy result for compactly supported potentials in the two dimensional case follows from the corresponding uniqueness result for the Dirichlet-to-Neumann map of Bukhgeim \cite{bukhgeim}.

For the magnetic Schr\"odinger equation with smooth magnetic and electric potentials with $n \geq 3$, it was shown in \cite{nakamurasunuhlmann} that the Dirichlet-to-Neumann map measured on the boundary of any domain determines uniquely the magnetic field and electric potential. The smoothness assumptions were relaxed in \cite{tolmasky}, \cite{salothesis}, \cite{salo}. These results imply the corresponding uniqueness theorem for the fixed energy scattering problem for compactly supported potentials, see \cite{eskinralston_ima} for this reduction in the magnetic case. Uniqueness for exponentially decaying potentials with $A \in C^{n+5}$ and $V \in C^{n+4}$ was proved in \cite{eskinralston} based on a method involving integral equations. A connection between the integral equations method and complex geometrical optics solutions is given in \cite{sun_exponential}. We also mention the work \cite{isozaki} which studies inverse scattering for Dirac operators, and \cite{weder}, \cite{wederyafaev} which consider fixed energy inverse scattering for short range potentials having regular behavior at infinity.

In this paper, the general outline for proving the uniqueness result is the same as in \cite{melrose} and \cite{uhlmannvasy} and consists of the following steps:

\begin{enumerate}
\item[1.] 
The scattering matrices coincide at a fixed energy. Thus one obtains, by using a version of Green's theorem at infinity, an integral identity relating the difference of potentials with products of scattering solutions of the Schr\"odinger equation.
\item[2.] 
The scattering solutions are dense in the space of all solutions with sufficiently small exponential growth, allowing to use such solutions in the integral identity.
\item[3.] 
Finally, one employs an analytic Fredholm argument to pass from solutions with small exponential growth to the complex geometrical optics solutions which may grow very rapidly. These solutions can be used to show that the coefficients are uniquely determined.
\end{enumerate}

The arguments for achieving Steps 1 and 2 above are very close to \cite{uhlmannvasy}, except that we have used the Agmon-H\"ormander spaces $B$ and $B^*$ to obtain more precise statements. However, the construction of suitable complex geometrical optics solutions is considerably more difficult in the magnetic case. This part involves the global version, established in \cite{salo} for compactly supported coefficients, of the pseudodifferential conjugation argument in \cite{nakamurauhlmann}. We note that a similar pseudodifferential conjugation also appears in \cite{eskinralston}. It is shown in Section \ref{sec:cgo_solutions} that the method in \cite{salo} can be extended to coefficients satisfying short range conditions, and we give a rather precise construction of complex geometrical optics solutions in this case.

The structure of the paper is as follows. Section \ref{sec:intro} is the introduction, and Section \ref{sec:preliminaries} contains well-known results for the direct scattering problem. However, since we could not find precise references for all results in the present setting, we give a rather careful account based on the exposition in \cite{H2}. In Section \ref{sec:cgo_solutions} we present the semiclassical pseudodifferential conjugation approach and analogs of the Sylvester-Uhlmann estimates \cite{sylvesteruhlmann} which allow to construct complex geometrical optics solutions for short range coefficients. The analytic Fredholm argument required to go from solutions with small exponential growth to complex geometrical optics solutions is presented in Section \ref{sec:analyticity}, and the uniqueness result, Theorem \ref{thm:maintheorem}, is proved in Section \ref{sec:uniqueness}.

\subsection*{Acknowledgements}

L.P.~and M.S.~are partly supported by the Academy of Finland, and G.U.~is supported in part by NSF and a Walker Family Endowed Professorship.

\section{Preliminaries} \label{sec:preliminaries}

In this section we recall some basic results in scattering theory related to the resolvent and scattering matrix for $H$. To obtain precise statements, the results will be formulated in terms of the Agmon-H\"ormander spaces $B$ and $B^*$. We refer to \cite[Chapter XIV]{H2} for more details on this approach.

\subsection{Function spaces}

The space $B$ (see \cite[Section 14.1]{H2}) is the set of those $u \in L^2(\mR^n)$ for which the norm 
\begin{equation*}
\norm{u}_B = \sum_{j=1}^{\infty} ( 2^{j-1} \int_{X_j} \abs{u}^2 \,dx )^{1/2}
\end{equation*}
is finite. Here $X_1 = \{ \abs{x} < 1 \}$ and $X_j = \{ 2^{j-2} < \abs{x} < 2^{j-1} \}$ for $j \geq 2$. This is a Banach space whose dual $B^*$ consists of all $u \in L^2_{\text{loc}}(\mR^n)$ such that 
\begin{equation*}
\norm{u}_{B^*} = \sup_{R > 1} \frac{1}{R} \int_{\abs{x} < R} \abs{u}^2 \,dx < \infty.
\end{equation*}
The set $C^{\infty}_c(\mR^n)$ is dense in $B$ but not in $B^*$. The closure in $B^*$ is denoted by $\mathring{B}^*$, and $u \in B^*$ belongs to $\mathring{B}^*$ iff 
\begin{equation*}
\lim_{R \to \infty} \frac{1}{R} \int_{\abs{x} < R} \abs{u}^2 \,dx = 0.
\end{equation*}
We will also need the Sobolev space variant $B^*_2$ of $B^*$, defined via the norm 
\begin{equation*}
\norm{u}_{B^*_2} = \sum_{\abs{\alpha} \leq 2} \norm{D^{\alpha} u}_{B^*}.
\end{equation*}
Let $L^2_{\delta}$ and $H^s_{\delta}$ be the weighted $L^2$ and Sobolev spaces in $\mR^n$ with norms 
\begin{equation*}
\norm{u}_{L^2_{\delta}} = \norm{\br{x}^{\delta} u}_{L^2}, \qquad \norm{u}_{H^s_{\delta}} = \norm{\br{x}^{\delta} u}_{H^s}.
\end{equation*}
Then one has $L^2_{\delta} \subseteq B \subseteq L^2_{1/2}$ and $L^2_{-1/2} \subseteq B^* \subseteq L^2_{-\delta}$ for any $\delta > 1/2$.

If $X$ is a function space and $a$ is a smooth positive function, we write $aX = \{ af \,;\, f \in X \}$ and $\norm{u}_{aX} = \norm{a^{-1} u}_X$. Note that $e^{\gamma \br{x}} H^k$ has equivalent norm $u \mapsto \sum_{\abs{\alpha} \leq k} \norm{e^{-\gamma \br{x}} \partial^{\alpha} u}_{L^2}$. For $u, v \in L^2(\mR^n)$ we use the sesquilinear pairing 
\begin{equation*}
(u|v) = \int_{\mR^n} u \bar{v} \,dx,
\end{equation*}
and we continue to use this notation if $u$ is in some function space and $v$ is in its dual.

If $\lambda > 0$ we will consider the sphere $M_{\lambda} = \{ \abs{\xi} = \sqrt{\lambda} \}$ with Euclidean surface measure $dS_{\lambda}$. The corresponding $L^2$ space is $L^2(M_{\lambda}) = L^2(M_{\lambda}, dS_{\lambda})$, and of course $L^2(S^{n-1}) = L^2(M_1)$.

The Fourier transform on functions in $\mR^n$ is defined by 
\begin{equation*}
\hat{f}(\xi) = \mF f(\xi) = \int_{\mR^n} e^{-ix \cdot \xi} f(x) \,dx,
\end{equation*}
and the inverse Fourier transform is 
\begin{equation*}
f(x) = \mF^{-1} \hat{f}(x) = (2\pi)^{-n} \int_{\mR^n} e^{ix \cdot \xi} \hat{f}(\xi) \,d\xi.
\end{equation*}

\subsection{Resolvents}

Let $H_0 = -\Delta$ be the free Schr\"odinger operator in $\mR^n$, and write $R_0(\lambda) = (H_0 - \lambda)^{-1}$ for the free resolvent if $\lambda \in \mC \smallsetminus [0,\infty)$. The limits $R_0(\lambda \pm i0)$ as $\lambda$ approaches the positive real axis are well defined and have the following properties.

\begin{definition}
We write 
\begin{equation*}
u \sim u_0
\end{equation*}
if $u$ has the same asymptotics as $u_0$ at infinity, meaning that $u = \psi(r) u_0 + u_1$ for some $u_1 \in \mathring{B}^*$. Here $\psi \in C^{\infty}(\mR)$ is a fixed function with $\psi(t) = 0$ for $\abs{t} \leq 1$ and $\psi(t) = 1$ for $\abs{t} \geq 2$, and we write $x = r\theta$ with $r \geq 0$ and $\theta \in S^{n-1}$.
\end{definition}

\begin{prop} \label{free_resolvent_estimates}
If $\lambda > 0$, then $R_0(\lambda \pm i0)$ is bounded $B \to B^*_2$ and 
\begin{equation} \label{free_resolvent_asymptotics}
R_0(\lambda \pm i0) f \sim c_{\lambda} (\pm 1)^{\frac{n+1}{2}} r^{-\frac{n-1}{2}} e^{\pm i\sqrt{\lambda} r} \hat{f}(\pm \sqrt{\lambda} \theta)
\end{equation}
where $c_{\lambda} = (\sqrt{\lambda}/2\pi i)^{\frac{n-3}{2}} / 4\pi$ and $(-1)^{\frac{n+1}{2}} = i^{n+1}$. If $u \in B^*$ is such that $(H_0-\lambda) u = f \in B$, then 
\begin{equation*}
u = u_{\pm} + R_0(\lambda \mp i0) f
\end{equation*}
where $u_{\pm} = \mF^{-1} \{ v_{\pm} \,dS_{\lambda} \}$ for some $v_{\pm} \in L^2(M_{\lambda})$.
\end{prop}
\begin{proof}
The boundedness follows from \cite[Section 14.3]{H2}. If $f \in B$ choose $f_j \in C^{\infty}_c(\mR^n)$ with $f_j \to f$ in $B$. Then by \cite[Section 1.7]{melrose}
\begin{equation*}
R_0(\lambda \pm i0) f_j = c_{\lambda} (\pm 1)^{\frac{n+1}{2}} \psi(r) r^{-\frac{n-1}{2}} e^{\pm i\sqrt{\lambda} r} \hat{f}_j(\pm \sqrt{\lambda} \theta) + v_j
\end{equation*}
where $v_j$ is smooth and $\abs{v_j(r\theta)} \leq C r^{-\frac{n-3}{2}}$, so that $v_j \in \mathring{B}^*$. Since $R_0(\lambda \pm i0)$ is bounded $B \to B^*$, and since the map 
\begin{equation*}
f \mapsto \hat{f}|_{M_{\lambda}}
\end{equation*}
is bounded $B \to L^2(M_{\lambda})$ by \cite[Theorem 14.1.1]{H2}, we have that $v_j$ converges to some $v \in \mathring{B}^*$ as $j \to \infty$. This proves \eqref{free_resolvent_asymptotics}, and the last part follows from \cite[Theorem 14.3.8]{H2}.
\end{proof}

An operator 
\begin{equation*}
V(x,D) = \sum_{\abs{\alpha} \leq 1} a_{\alpha}(x) D^{\alpha},
\end{equation*}
where $a_{\alpha} \in L^2_{\text{loc}}(\mR^n)$, is said to be \emph{short range} if $V(x,D)$ (initially defined on $C^{\infty}_c(\mR^n)$) extends to a compact operator from $B^*_2$ into $B$. A sufficient condition for $V(x,D)$ to be short range is that there exists a decreasing function $M: [0,\infty) \to [0,\infty)$ with $\int_0^{\infty} M(t) \,dt < \infty$, such that 
\begin{equation*}
\abs{a_{\alpha}(x)} \leq M(\abs{x}), \quad \text{for all } \abs{\alpha} \leq 1.
\end{equation*}
We will assume that $V(x,D)$ is short range and also symmetric, that is, 
\begin{equation*}
(V(x,D) u|v) = (u|V(x,D) v), \quad u, v \in C^{\infty}_c(\mR^n).
\end{equation*}
For symmetric short range $V$ the last identity remains true for $u, v \in B^*_2$ (see \cite[Section 14.4]{H2}). Any such $V(x,D)$ must be of the magnetic form 
\begin{equation*}
V(x,D) = 2A \cdot D + A^2 + D \cdot A + V \quad \text{ with $A(x)$, $V(x)$ real}.
\end{equation*}

The perturbed Hamiltonian is given by 
\begin{equation*}
H = H_0 + V(x,D).
\end{equation*}
If $V(x,D)$ is symmetric and short range then $H$, with domain $\mS(\mR^n)$, is essentially self-adjoint on $L^2(\mR^n)$ \cite[Theorem 14.4.4]{H2}.

The resolvent of $H$ on the real axis is given by 
\begin{equation*}
R(\lambda \pm i0) = R_0(\lambda \pm i0) (I+V(x,D) R_0(\lambda \pm i0))^{-1}, \quad \lambda > 0, \ \lambda \notin \Lambda.
\end{equation*}
Here $\Lambda = \{ \lambda \in \mR \smallsetminus \{0\} \,;\, (H-\lambda)u = 0 \text{ for some } u \in L^2, u \neq 0 \}$ is the set of eigenvalues. By \cite[Section 14.5]{H2} the operator $I+V(x,D) R_0(\lambda \pm i0)$ is invertible on $B$ when $\lambda > 0$ is not an eigenvalue, and the resolvent maps $B$ to $B_2^*$. If 
\begin{equation} \label{ucp_condition}
\abs{a_{\alpha}(x)} \leq C \br{x}^{-1}, \quad \abs{\alpha} \leq 1,
\end{equation}
then the fact that any eigenfunction is rapidly decreasing \cite[Theorem 14.5.5]{H2} and unique continuation at infinity \cite[Theorem 5.2]{hormander_ucp} show that there are no positive eigenvalues and the resolvent is defined for all $\lambda > 0$.

\subsection{Scattering matrix}

Let $V(x,D)$ be a symmetric short range perturbation. By \cite[Theorem 14.4.6]{H2} the wave operators, defined by 
\begin{equation*}
W_{\pm} u = \lim_{t \to \pm \infty} e^{itH} e^{-it H_0} u, \quad u \in L^2(\mR^n),
\end{equation*}
exist as strong limits and are isometric operators intertwining the closures of $H$ and $H_0$. Their range is the projection of $L^2$ onto continuous spectrum, and the scattering operator 
\begin{equation*}
S = W_+^* W_-
\end{equation*}
is a unitary operator on $L^2(\mR^n)$ \cite[Theorem 14.6.5]{H2}.

Let $\lambda > 0$. It follows from \cite[Theorem 14.6.8]{H2} that there exists a unitary map $\Sigma_{\lambda}$ on $L^2(M_{\lambda}, dS_{\lambda}/2\sqrt{\lambda})$ such that for $f \in L^2(\mR^n)$ one has 
\begin{equation*}
(\mF S \mF^{-1}) f |_{M_{\lambda}} = \Sigma_{\lambda}(f|_{M_{\lambda}})
\end{equation*}
for almost every $\lambda$. The map $\Sigma_{\lambda}$ is the \emph{scattering matrix} at energy $\lambda$.

It will be convenient to describe the scattering matrix in terms of the Poisson operators. The free Poisson operator acts on functions $g \in L^2(M_{\lambda})$ by 
\begin{equation*}
P_0(\lambda) g(x) = \frac{i}{(2\pi)^{n-1}} \int_{M_{\lambda}} e^{i x \cdot \xi} g(\xi) \,\frac{dS_{\lambda}(\xi)}{2\sqrt{\lambda}}, \qquad \lambda > 0,
\end{equation*}
and the (outgoing) Poisson operator for $H$ is defined by 
\begin{equation*}
P(\lambda)g = P_0(\lambda) g - R(\lambda + i0)(V(x,D) P_0(\lambda) g).
\end{equation*}
These operators map $L^2(M_{\lambda})$ continuously into $B_2^*$ by \cite[Theorem 7.1.26]{H2} and the above discussion, and one has $(H_0-\lambda) P_0(\lambda) g = (H-\lambda) P(\lambda) g = 0$. Stationary phase \cite[Theorem 7.7.14]{H2} and an approximation argument imply the asymptotics (with $c_{\lambda}$ given in Proposition \ref{free_resolvent_asymptotics}) 
\begin{equation} \label{free_poisson_asymptotics}
P_0(\lambda) g \sim c_{\lambda} r^{-\frac{n-1}{2}} \left[ e^{i\sqrt{\lambda} r} g(\sqrt{\lambda} \theta) + i^{n-1} e^{-i\sqrt{\lambda} r} g(-\sqrt{\lambda} \theta) \right].
\end{equation}

\begin{prop} \label{prop:poisson_scattering_matrix}
If $g \in L^2(M_{\lambda})$ then 
\begin{equation} \label{perturbed_poisson_asymptotics}
P(\lambda) g \sim c_{\lambda} r^{-\frac{n-1}{2}} \left[ e^{i\sqrt{\lambda} r} (\Sigma_{\lambda} g)(\sqrt{\lambda} \theta) + i^{n-1} e^{-i\sqrt{\lambda} r} g(-\sqrt{\lambda}\theta) \right].
\end{equation}
\end{prop}
\begin{proof}
Let $g \in L^2(M_{\lambda})$, and define $v_- = g$. As in \cite[Theorem 14.6.8]{H2} define 
\begin{equation} \label{uminus_poisson}
u_- = \mF^{-1} \{ v_- \delta_0(\abs{\xi}^2-\lambda) \} = \frac{1}{2\pi i} P_0(\lambda) g.
\end{equation}
Here by \cite[Theorem 6.1.5]{H2} one has $\delta_0(\abs{\xi}^2-\lambda) = dS_{\lambda}/2\sqrt{\lambda}$.

Let $u \in B_2^*$ satisfy $u = u_- - R_0(\lambda+i0) Vu$. Such a function $u$ exists by the proof of \cite[Theorem 14.6.8]{H2}, and then 
\begin{align} \label{u_poisson}
u &= u_- - R_0(\lambda+i0)Vu = u_- - R(\lambda+i0)V(I+R_0(\lambda+i0)V)u \notag \\
 &= \frac{1}{2\pi i} P(\lambda) g.
\end{align}
Define $u_+ = u + R_0(\lambda-i0)Vu$. Since $u_+ \in B^*$ and $(H_0-\lambda)u_+ = 0$, Proposition \ref{free_resolvent_estimates} shows that $u_+ = \mF^{-1}(v_+ \delta_0(\abs{\xi}^2-\lambda) \}$ for some $v_+$ in $L^2(M_{\lambda})$. By \cite[Theorem 14.6.8]{H2}, one has 
\begin{equation*}
\Sigma_{\lambda} v_- = v_+
\end{equation*}
and therefore 
\begin{equation} \label{uplus_poisson}
u_+ = \frac{1}{2\pi i} P_0(\lambda) \Sigma_{\lambda} g.
\end{equation}

Combining \eqref{u_poisson} with the representations of $u$ in terms of $u_-$ and $u_+$ and using \eqref{uminus_poisson} and \eqref{uplus_poisson}, we obtain 
\begin{align*}
P(\lambda) g &= P_0(\lambda) g - R_0(\lambda+i0) w, \\
P(\lambda) g &= P_0(\lambda) \Sigma_{\lambda} g - R_0(\lambda-i0) \tilde{w}
\end{align*}
for some $w, \tilde{w} \in B$. The result follows from \eqref{free_resolvent_asymptotics}, \eqref{free_poisson_asymptotics}, and the uniqueness of such asymptotics.
\end{proof}

The next result gives an analog of Green's theorem for this setting. This is the ''boundary pairing'' appearing in \cite{melrose}, \cite{uhlmannvasy}.

\begin{prop} \label{prop:boundary_pairing}
Assume $u, v \in B^*$ and $(H_0-\lambda)u \in B, (H_0-\lambda)v \in B$. If $u$ and $v$ have the asymptotics 
\begin{align}
u &\sim r^{-\frac{n-1}{2}} \left[ e^{i\sqrt{\lambda} r} g_+(\theta) + e^{-i\sqrt{\lambda} r} g_-(\theta) \right], \label{bp_u_asymptotic} \\
v &\sim r^{-\frac{n-1}{2}} \left[ e^{i\sqrt{\lambda} r} h_+(\theta) + e^{-i\sqrt{\lambda} r} h_-(\theta) \right] \label{bp_v_asymptotic}
\end{align}
for some $g_{\pm}, h_{\pm} \in L^2(S^{n-1})$, then 
\begin{equation*}
(u|(H_0-\lambda)v) - ((H_0-\lambda)u|v) = 2i\sqrt{\lambda} \left[ (g_+|h_+)_{S^{n-1}} - (g_-|h_-)_{S^{n-1}} \right].
\end{equation*}
Here $(g|h)_{S^{n-1}} = \int_{S^{n-1}} g \bar{h} \,dS$.
\end{prop}
\begin{proof}
Write $f = (H_0-\lambda)u$, and let $u = \mF^{-1}\{ v_+ \,dS_{\lambda} \} + R_0(\lambda-i0)f$ as in Proposition \ref{free_resolvent_estimates}. Choose $v_{+j} \in C^{\infty}(M_{\lambda})$ and $f_j \in C^{\infty}_c(\mR^n)$ with $v_{+j} \to v_+$ in $L^2(M_{\lambda})$ and $f_j \to f$ in $B$, and write $u_j = \mF^{-1}\{ v_{+j} \,dS_{\lambda} \} + R_0(\lambda-i0)f_j$. Then $u_j \to u$ in $B^*$ and $(H_0-\lambda)u_j \to (H_0-\lambda)u$ in $B$. By \eqref{free_resolvent_asymptotics} and \eqref{free_poisson_asymptotics} 
\begin{equation*}
u_j = \psi(r) r^{-\frac{n-1}{2}} \left[ e^{i\sqrt{\lambda} r} g_{+j}(\theta) + e^{-i\sqrt{\lambda} r} g_{-j}(\theta) \right] + w_j
\end{equation*}
where $g_{\pm j} \in C^{\infty}(S^{n-1})$ converge to $g_{\pm}$ in $L^2(S^{n-1})$, and $w_j$ are smooth functions with $\abs{w_j} \leq C r^{-\frac{n-3}{2}}$ converging to some $w$ in $\mathring{B}^*$. In fact, since $f_j$ and $g_{\pm j}$ are smooth one also has the estimate $\abs{\nabla w_j} \leq C r^{-\frac{n-3}{2}}$ (see \cite[Sections 1.3 and 1.7]{melrose}).

Performing a similar approximation for $v$, we integrate by parts in a ball $B(0,R)$ to obtain 
\begin{equation*}
(u_j|(H_0-\lambda)v_j)_{B(0,R)} - ((H_0-\lambda)u_j|v_j)_{B(0,R)} = \int_{\partial B(0,R)} \left( \frac{\partial u_j}{\partial \nu} \bar{v}_j - u_j \frac{\partial \bar{v}_j}{\partial \nu} \right) \,dS.
\end{equation*}
Here 
\begin{equation*}
\frac{\partial u_j}{\partial \nu}(r\theta) = i\sqrt{\lambda} \psi(r) r^{-\frac{n-1}{2}} \left[ e^{i\sqrt{\lambda} r} g_{+j}(\theta) - e^{-i\sqrt{\lambda} r} g_{-j}(\theta) \right] + \tilde{w}_j
\end{equation*}
with $\abs{\tilde{w}_j} \leq C r^{-\frac{n-3}{2}}$, and similarly for $v_j$. Inserting the asymptotics in the boundary term and letting $R \to \infty$ gives 
\begin{equation*}
(u_j|(H_0-\lambda)v_j) - ((H_0-\lambda)u_j|v_j) = 2i\sqrt{\lambda} \left[ (g_{+j}|h_{+j})_{S^{n-1}} - (g_{-j}|h_{-j})_{S^{n-1}}\right].
\end{equation*}
The result follows upon taking $j \to \infty$.
\end{proof}

The last result here concerns the density of the scattering solutions $P(\lambda)g$ in the set of all exponentially growing solutions, as in \cite{uhlmannvasy}. The result is valid for exponentially decaying coefficients.

\begin{prop} \label{prop:density}
Let $V(x,D)$ be a symmetric short range perturbation with 
\begin{equation*}
\abs{a_{\alpha}(x)} \leq C e^{-\gamma_0 \br{x}}, \quad \abs{\alpha} \leq 1,
\end{equation*}
for some $\gamma_0 > 0$. Let $0 < \gamma < \gamma_0$. Given any $w \in e^{\gamma \br{x}} L^2$ with $(H-\lambda)w = 0$, there exist $g_j \in L^2(M_{\lambda})$ such that $P(\lambda) g_j \to w$ in $e^{\gamma_0 \br{x}} L^2$.
\end{prop}
\begin{proof}
Suppose $f \in e^{-\gamma_0 \br{x}} L^2$ and 
\begin{equation*}
(u|f) = 0
\end{equation*}
for any $u = P(\lambda)g$ where $g \in L^2(M_{\lambda})$. Let $w \in e^{\gamma \br{x}} L^2$ with $(H-\lambda)w = 0$. We need to show that $(w|f) = 0$, which will imply that $w$ is in the closure of the subspace $\{P(\lambda)g \,;\, g \in L^2(M_{\lambda}) \}$ of $e^{\gamma_0\br{x}} L^2$ as required.

Write $v = R(\lambda-i0)f$, so that $v = R_0(\lambda-i0)f'$ where $f'$ is the solution of $(I+VR_0(\lambda-i0))f' = f$. The operator $V R_0(\lambda-i0)$ maps $e^{-\gamma' \br{x}}L^2$ compactly into itself for any $\gamma' < \gamma_0$, which shows that $f' \in e^{-\gamma'\br{x}} L^2$ for all such $\gamma'$.

For $g \in L^2(M_{\lambda})$, let $u = P(\lambda)g$. Then 
\begin{equation*}
0 = (u|f) = (u|(H-\lambda)v) = (u|(H-\lambda)v) - ((H-\lambda)u|v).
\end{equation*}
Since $V$ is symmetric we may replace $H$ by $H_0$ in the last part, and then Proposition \ref{prop:boundary_pairing} implies 
\begin{equation*}
(g_+|h_+)_{S^{n-1}} - (g_-|h_-)_{S^{n-1}} = 0
\end{equation*}
if $g_{\pm}$ and $h_{\pm}$ are as in \eqref{bp_u_asymptotic} and \eqref{bp_v_asymptotic}. But $h_+ = 0$ by \eqref{free_resolvent_asymptotics}, and since $g$ was arbitrary we obtain from \eqref{perturbed_poisson_asymptotics} and \eqref{free_resolvent_asymptotics} that $h_- = \widehat{f'}(-\sqrt{\lambda} \,\cdot\,) = 0$. Now Lemma \ref{resolvent_exponential} below shows that $v \in e^{-\gamma' \br{x}} H^2$ for any $\gamma' < \gamma_0$.

Choose $v_j \in C^{\infty}_c(\mR^n)$ with $v_j \to v$ in $e^{-\gamma \br{x}} H^2$. Then 
\begin{equation*}
(w|f) = (w|(H-\lambda)v) = \lim_{j \to \infty} (w|(H-\lambda)v_j) = \lim_{j \to \infty} ((H-\lambda)w|v_j) = 0
\end{equation*}
as required.
\end{proof}

\begin{lemma} \label{resolvent_exponential}
If $f \in e^{-\gamma \br{x}} L^2$ for $\gamma < \gamma_0$, and if $\hat{f}|_{M_{\lambda}} = 0$, then one has $R_0(\lambda \pm i0) f \in e^{-\gamma \br{x}} H^2$ for any $\gamma < \gamma_0$.
\end{lemma}
\begin{proof}
Let $U = \{ z \in \mC^n \,;\, \abs{ \im\,z} < \gamma_0 \}$. Then the Paley-Wiener theorem for exponentially decaying functions, \cite[Theorem IX.13]{reedsimon}, shows that $\hat{f}$ extends as an analytic function in $U$ with 
\begin{equation} \label{f_exponential_bounds}
\sup_{\abs{\eta} \leq \gamma} \norm{\hat{f}(\,\cdot\,+i\eta)}_{L^2} < \infty \qquad \text{for all } \gamma < \gamma_0.
\end{equation}
Define $M_{\lambda}^{\mC} = \{ z \in U \,;\, z \cdot z = \lambda \}$. This is a complex submanifold of $U$ of complex codimension one, and since $\hat{f}$ vanishes on the real zeros of $z \cdot z - \lambda$ we have 
\begin{equation*}
\hat{f}(z) = (z \cdot z - \lambda) g(z), \quad z \in U,
\end{equation*}
for some function $g$ analytic in $U$.

To see the last claim, let $z^0$ be any point in $M_{\lambda}^{\mC} \cap \mR^n$. Then some coordinate, say $z_n^0$, is nonzero and the map $\varphi(z',z_n) = (z',z \cdot z - \lambda)$ is a biholomorphic diffeomorphism defined on a neighborhood $W$ of $z^0$ in $U$ by the inverse function theorem. Since $\hat{f} \circ \varphi^{-1}$ vanishes on all real points $(\xi',0)$ in $\varphi(W)$, it vanishes on all points $(z',0)$ in $\varphi(W)$ and thus $\hat{f}$ vanishes on $M_{\lambda}^{\mC} \cap W$. Now $\hat{f}$ is analytic on the connected manifold $M_{\lambda}^{\mC}$, so it vanishes on this manifold. Using a corresponding biholomorphic map $\varphi$ near any point of $M_{\lambda}^{\mC}$ shows that $\hat{f}/(z \cdot z - \lambda)$ is locally bounded in $U \smallsetminus M_{\lambda}^{\mC}$, so the required function $g$ exists by the removable singularities theorem in several complex variables.

Let $\abs{\alpha} \leq 2$. The fact that $\hat{f}(\xi)$ vanishes on $M_{\lambda}$ implies 
\begin{equation*}
(D^{\alpha} R_0(\lambda \pm i0) f)\ehat(\xi) = \frac{\xi^{\alpha}}{\xi \cdot \xi - \lambda} \hat{f}(\xi), \quad \xi \in \mR^n.
\end{equation*}
We define 
\begin{equation*}
h_{\alpha}(z) = z^{\alpha} g(z) =  \frac{z^{\alpha}}{z \cdot z - \lambda} \hat{f}(z), \quad z \in U.
\end{equation*}
This is an analytic function in $U$, hence bounded on compact subsets of $U$, and if $\abs{\im\,z} \leq \gamma < \gamma_0$ then $z^{\alpha}/(z \cdot z - \lambda)$ is bounded for $\abs{\re\,z}$ large. By the estimates \eqref{f_exponential_bounds} we see $h_{\alpha}$ satisfies similar estimates. The result now follows from \cite[Theorem IX.13]{reedsimon}.
\end{proof}

\subsection{Magnetic Schr\"odinger operator}

We now specialize to the case of the magnetic Schr\"odinger operator with exponentially decaying coefficients. More precisely, suppose that $A$ and $V$ satisfy \eqref{av_assumptions}, and let $H$ be the corresponding magnetic Schr\"odinger operator. Clearly $H = H_0 + V(x,D)$ where 
\begin{equation*}
V(x,D) = 2 A \cdot D + \tilde{V}, \qquad \tilde{V}(x) = A^2 + D \cdot A + V(x).
\end{equation*}
Thus $V(x,D)$ is a symmetric short range perturbation satisfying \eqref{ucp_condition}, so $H$ has no positive eigenvalues and the resolvent $R(\lambda \pm i0)$ is well defined for all $\lambda > 0$.

Next we show that the scattering matrix for $H$ is preserved under gauge transformations.

\begin{lemma} \label{lemma:gaugeinvariance}
If $\alpha \in e^{-\gamma_0 \br{x}} W^{2,\infty}(\mR^n ; \mR)$, then the scattering matrices for the coefficients $(A,V)$ and $(A+\nabla \alpha,V)$ are equal.
\end{lemma}
\begin{proof}
We will use Proposition \ref{prop:poisson_scattering_matrix}. Writing $H'$ for the Schr\"odinger operator with coefficients $(A+\nabla \alpha, V)$, we have 
\begin{equation*}
H' = e^{-i\alpha} H e^{i\alpha}.
\end{equation*}
Then $u' = e^{-i\alpha} P(\lambda) g$ with $g \in L^2(M_{\lambda})$ solves $(H'-\lambda)u' = 0$, and by \eqref{perturbed_poisson_asymptotics}
\begin{align*}
u' &= P(\lambda)g + (e^{-i\alpha}-1)P(\lambda)g \\
 &\sim c_{\lambda} r^{-\frac{n-1}{2}} \left[ e^{i\sqrt{\lambda} r} (\Sigma_{\lambda} g)(\sqrt{\lambda} \theta) + i^{n-1} e^{-i\sqrt{\lambda} r} g(-\sqrt{\lambda}\theta) \right].
\end{align*}
We have $(H'-\lambda)(u'-P'(\lambda)g) = 0$, and by \eqref{perturbed_poisson_asymptotics} again 
\begin{equation*}
u' - P'(\lambda)g \sim c_{\lambda} r^{-\frac{n-1}{2}} e^{i\sqrt{\lambda} r} \left[ (\Sigma_{\lambda} g)(\sqrt{\lambda} \theta) - (\Sigma_{\lambda}' g)(\sqrt{\lambda} \theta) \right].
\end{equation*}
Then \cite[Lemma 14.6.6]{H2} and \eqref{free_resolvent_asymptotics}, \eqref{free_poisson_asymptotics} show that $\Sigma_{\lambda} g = \Sigma_{\lambda}' g$.
\end{proof}

Let now $H'$ be another magnetic Schr\"odinger operator with coefficients $A'$ and $V'$ satisfying \eqref{av_assumptions}. The following orthogonality identity for exponentially growing solutions will be used in recovering the coefficients.

\begin{lemma} \label{lemma:boundarypairing}
If $\Sigma_{\lambda} = \Sigma_{\lambda}'$, then 
\begin{equation*}
((2A \cdot D + \tilde{V})w|w') - (w|(2A' \cdot D + \tilde{V}')w') = 0
\end{equation*}
for all $w, w' \in e^{\gamma \br{x}} H^1$ with $(H-\lambda)w = 0$, $(H'-\lambda)w' = 0$, and $\gamma < \frac{\gamma_0}{2}$. Further, if $A = A'$, then 
\begin{equation*}
(Vw|w') - (w|V'w') = 0
\end{equation*}
for all such $w$, $w'$.
\end{lemma}
\begin{proof}
Let $g, g' \in L^2(M_{\lambda})$ and choose $u = P(\lambda)g$ and $u' = P'(\lambda)g'$ where $g' = (\Sigma_{\lambda}')^* \tilde{g}'$. The functions $u$ and $u'$ satisfy the conditions of Proposition \ref{prop:boundary_pairing}, and we obtain by \eqref{perturbed_poisson_asymptotics} and the unitarity of $\Sigma_{\lambda}'$ that 
\begin{equation*}
(V(x,D)u|u') - (u|V'(x,D)u') = 2i c_{\lambda}^2 \sqrt{\lambda} ((\Sigma_{\lambda} - \Sigma_{\lambda}')g(\sqrt{\lambda} \,\cdot\,)|\tilde{g}'(\sqrt{\lambda} \,\cdot\,))_{S^{n-1}}.
\end{equation*}
Therefore 
\begin{equation} \label{u_scattering_orthogonality}
((2A \cdot D + \tilde{V})u|u') - (u|(2A' \cdot D + \tilde{V}')u') = 0.
\end{equation}
The density result in Proposition \ref{prop:density} implies that we can find scattering solutions $u_j$ and $u_j'$ so that $u_j \to w$ and $u_j' \to w'$ in $e^{\gamma \br{x}} L^2$. Then by \eqref{av_assumptions} $(2 A \cdot D + \tilde{V}) u_j \to (2A \cdot D + \tilde{V}) w$ in $e^{-\gamma \br{x}} H^{-1}$ and similarly for $w'$, and we obtain the first identity by taking suitable limits in \eqref{u_scattering_orthogonality}.

If $A = A'$, \eqref{u_scattering_orthogonality} gives upon integrating by parts in a large ball and using the exponential decay of $A$ that 
\begin{equation*}
(Vu|u') - (u|V'u') = 0
\end{equation*}
for all scattering solutions $u$ and $u'$. Approximation yields the same identity for $w$ and $w'$.
\end{proof}

Finally, we record a result which will guarantee the existence of an appropriate gauge transformation.

\begin{lemma} \label{lemma:gauge_decay}
Let $A \in e^{-\gamma\br{x}} W^{1,\infty}(\mR^n ; \mR^n)$ with $dA = 0$. Then $A = \nabla \alpha$ for some $\alpha \in e^{-\gamma\br{x}} W^{2,\infty}(\mR^n ; \mR)$.
\end{lemma}
\begin{proof}
We define 
\begin{equation*}
\alpha(x) = \int_0^1 x \cdot A(tx) \,dt.
\end{equation*}
Then $\alpha$ is Lipschitz, and $\nabla \alpha = A$ follows by a direct computation using the fact that $dA = 0$. If $r > 0$ and $\omega \in S^{n-1}$, we have 
\begin{equation*}
\alpha(r\omega) = \int_0^r \omega \cdot A(t\omega) \,dt = \int_{[0,r\omega]} A
\end{equation*}
where the last integral is the integral of a $1$-form over the line segment from $0$ to $r\omega$. Since $dA = 0$, such an integral over a closed curve vanishes by the Stokes theorem, and we have 
\begin{equation*}
\lim_{r \to \infty} (\alpha(r\omega_1) - \alpha(r\omega_2)) = \lim_{r \to \infty} \int_{[r\omega_2,r\omega_1]} A = 0
\end{equation*}
for any $\omega_1,\omega_2 \in S^{n-1}$. Therefore, $\lim_{r \to \infty} \alpha(r\omega)$ is a constant independent of $\omega$, and by substracting a constant from $\alpha$ we may assume that this limit is $0$. We may now write 
\begin{equation*}
\alpha(r\omega) = -\int_r^{\infty} \omega \cdot A(t\omega) \,dt,
\end{equation*}
and an easy computation gives that $\alpha \in e^{-\gamma\br{x}} W^{2,\infty}$.
\end{proof}

\section{Complex geometrical optics solutions} \label{sec:cgo_solutions}

Instead of the scattering solutions to $(H-\lambda)u = 0$ considered in the previous section, we want use use solutions with "complex frequency" to recover the coefficients from the scattering matrix. These will be the complex geometrical optics solutions introduced by Sylvester-Uhlmann \cite{sylvesteruhlmann}, based on earlier work of Calder{\'o}n \cite{calderon}.

In this section we assume that $A$ and $V$ satisfy 
\begin{eqnarray}
 & A \in \br{x}^{-1-\varepsilon} C_b(\mR^n ; \mC^n), \quad \nabla \cdot A \in \br{x}^{-1} L^n, & \label{w_cgo_assumption} \\
 & V \in \br{x}^{-1} L^n(\mR^n ; \mC), & \label{v_cgo_assumption}
\end{eqnarray}
for some $\varepsilon > 0$ (we write $C_b$ for the bounded continuous functions). For the main result we also assume 
\begin{equation} \label{ltwo_cgo_assumption}
\norm{\br{x} A}_{L^2} + \norm{\br{x} \tilde{V}}_{L^2} < \infty.
\end{equation}
Here and below we will write 
\begin{equation*}
\tilde{V} = A^2 + D \cdot A + V.
\end{equation*}
We consider solutions to $(H-\lambda)u = 0$ of the form 
\begin{equation} \label{urho_def}
u_{\rho} = e^{i\rho \cdot x}(1+v_{\rho})
\end{equation}
where $\rho \in \mC^n$ satisfies $\rho \cdot \rho = \lambda$, and $v_{\rho} \in H^1_{\delta}$ where $-1 < \delta < 0$. The main point is that when $\abs{\rho}$ is sufficiently large, these complex geometrical optics solutions exist and the asymptotic behaviour of $v_{\rho}$ as $\abs{\rho} \to \infty$ is known.

We introduce some notation. Consider the conjugated operator 
\begin{equation} \label{conjugated_operator}
e^{-i\rho \cdot x} (H-\lambda) e^{i\rho \cdot x} = P_{\rho} + 2A \cdot D_{\rho} + \tilde{V}
\end{equation}
where 
\begin{equation*}
P_{\rho} = -\Delta + 2\rho \cdot D, \qquad D_{\rho} = D + \rho.
\end{equation*}
The operator $P_{\rho}$ has a right inverse $G_{\rho}$, whose mapping properties are well known.

\begin{prop} \label{prop:grho_estimates}
Let $-1 < \delta < 0$, and let $\rho \in \mC^n$ with $\rho \cdot \rho = \lambda$ and $\abs{\rho} \geq 1$. Then for any $f \in L^2_{\delta+1}$ the equation $P_{\rho} u = f$ has a unique solution $u \in L^2_{\delta}$. The solution operator $G_{\rho}: f \mapsto u$ satisfies 
\begin{equation*}
\norm{\partial^{\alpha} G_{\rho} f}_{L^2_{\delta}} \leq C \abs{\rho}^{\abs{\alpha}-1} \norm{f}_{L^2_{\delta+1}}
\end{equation*}
whenever $\abs{\alpha} \leq 2$.
\end{prop}
\begin{proof}
This is proved in \cite{sylvesteruhlmann} in the case $\alpha = 0$ and $\lambda = 0$, but the same proof works for $\lambda \neq 0$. See \cite{salo} for the simple extension to $\abs{\alpha} \leq 2$.
\end{proof}

We will write 
\begin{equation*}
K_{\rho} = (2A \cdot D_{\rho} + \tilde{V}) G_{\rho}.
\end{equation*}
The following is the main result. The formulation may look complicated, but the result is stated so that one only needs to know the statement of this proposition in the later sections.

\begin{prop} \label{prop:cgo_solutions_main}
Let $-1 < \delta < 0$ and suppose $\rho \in \mC^n$ with $\rho \cdot \rho = \lambda > 0$. Assume \eqref{w_cgo_assumption} -- \eqref{ltwo_cgo_assumption}. If $\abs{\rho}$ is sufficiently large, the equation $(H-\lambda)u = 0$ has a unique solution $u = u_{\rho}$ of the form \eqref{urho_def} where $v_{\rho} \in H^1_{\delta}$. In fact 
\begin{equation} \label{vkrho_def}
v_{\rho} = G_{\rho} (I+K_{\rho})^{-1} (-2A \cdot \rho - \tilde{V}),
\end{equation}
where $I+K_{\rho}$ is bounded and invertible on $L^2_{\delta+1}$, and the norm of $(I+K_{\rho})^{-1}$ is uniformly bounded for $\abs{\rho}$ large. Further, if 
\begin{equation} \label{rhoh_def}
\rho = \rho(h) = h^{-1}(\nu_1 + i(1-h^2\lambda)^{1/2} \nu_2)
\end{equation}
where $\nu_1, \nu_2 \in \mR^n$ are orthogonal unit vectors and $h$ is small, and if 
\begin{equation} \label{phi_cgo}
\phi(x) = -\frac{1}{2\pi} \int_{\mR^2} \frac{1}{y_1+i y_2} (\nu_1 + i\nu_2) \cdot A(x-y_1 \nu_1 - y_2 \nu_2) \,dy_1 \,dy_2,
\end{equation}
then one has the asymptotics $1+v_{\rho} = a_{\rho} + r_{\rho}$ where 
\begin{eqnarray}
 & \norm{a_{\rho}}_{L^{\infty}} = O(1), \quad h \norm{\nabla a_{\rho}}_{L^{\infty}} = o(1), \quad a_{\rho} \to e^{i \phi} \text{ pointwise}, & \label{arho_asymptotics} \\
 & \norm{r_{\rho}}_{L^2_{\delta}} + h \norm{\nabla r_{\rho}}_{L^2_{\delta}} = o(1), & \label{rrho_asymptotics}
\end{eqnarray}
as $h \to 0$.
\end{prop}

The rest of this section is devoted to proving Proposition \ref{prop:cgo_solutions_main}. Inserting \eqref{urho_def} in the equation $(H-\lambda)u = 0$, we see that to obtain complex geometrical optics solutions, it is enough to solve the conjugated equation 
\begin{equation} \label{conjugated_equation}
e^{-i\rho \cdot x} (H-\lambda) e^{i\rho \cdot x} v_{\rho} = f
\end{equation}
for a certain right hand side $f$. Most of the work will be to establish the following estimates for this equation.

\begin{prop} \label{prop:cgosolutions}
Let $-1 < \delta < 0$, and suppose $\rho \in \mC^n$ satisfies $\rho \cdot \rho = \lambda$. If \eqref{w_cgo_assumption} -- \eqref{v_cgo_assumption} hold and $\abs{\rho}$ is sufficiently large, then for any $f \in L^2_{\delta+1}$ the equation \eqref{conjugated_equation} has a unique solution $v \in H^1_{\delta}$. Further, $v \in H^2_{\delta}$, and 
\begin{equation*}
\norm{\partial^{\alpha} v}_{L^2_{\delta}} \leq C \abs{\rho}^{\abs{\alpha}-1} \norm{f}_{L^2_{\delta+1}}
\end{equation*}
whenever $\abs{\alpha} \leq 2$.
\end{prop}

In the case where $A \in C_c$ and $V \in L^{\infty}_c$ this was proved in \cite{salo} by using conjugation with semiclassical pseudodifferential operators. The conjugation method is due to Nakamura and Uhlmann \cite{nakamurauhlmann} in the context of inverse boundary value problems, and to Eskin and Ralston \cite{eskinralston} in inverse scattering problems. In \cite{salo} the method was extended to yield global solutions in weighted Sobolev spaces, and to handle nonsmooth coefficients. The proof involves a smoothing procedure and also a cutoff argument as in Kenig-Ponce-Vega \cite{kenigponcevega}. The proof of Proposition \ref{prop:cgosolutions} is parallel to that of \cite[Theorem 1.1]{salo}, except for the modifications needed because of coefficients which are not compactly supported.

We use the notation \eqref{conjugated_operator} and find the solution to \eqref{conjugated_equation} by a perturbation argument. We try $v = G_{\rho} w$, so $w$ must satisfy 
\begin{equation*}
(I + K_{\rho}) w = f.
\end{equation*}
We note that $K_{\rho}$ is bounded and compact on $L^2_{\delta+1}$ but it may not be small in norm. Therefore, one can not directly invert $I + K_{\rho}$ by Neumann series. Also, a possible Fredholm theory argument might not give the required uniform norm bound for $(I+K_{\rho})^{-1}$ for $\abs{\rho}$ large. To avoid these problems, we conjugate the original equation by pseudodifferential operators so that one obtains a norm small perturbation, which can be inverted by Neumann series.

It will be convenient to switch to semiclassical notation, since this automatically keeps track of the dependence of the norm estimates on $\abs{\rho}$. Thus, let $h = (\frac{1}{2}(\abs{\rho}^2+\lambda))^{-1/2}$ be the small parameter, and let 
\begin{align*}
P &= -h^2 \Delta + 2 \hat{\rho} \cdot hD, \\
Q &= 2A \cdot (hD + \hat{\rho})
\end{align*}
where $\hat{\rho} = h\rho$. Using $\rho \cdot \rho = \lambda$, there are orthogonal unit vectors $\nu_1, \nu_2 \in \mR^n$ so that 
\begin{equation*}
\hat{\rho} = \nu_1 + i (1-h^2 \lambda)^{1/2} \nu_2.
\end{equation*}
Here we assume $h$ small enough so $1-h^2 \lambda > 0$.

We will use the usual semiclassical symbol classes, see \cite{dimassisjostrand}.
\begin{definition}
If $0 \leq \sigma < 1/2$ and $m \in \mR$, we let $S^m_{\sigma}$ be the space of all functions $c(x,\xi) = c(x,\xi;h)$ where $x,\xi \in \mR^n$ and $h \in (0,h_0]$, $h_0 \leq 1$, such that $c$ is smooth in $x$ and $\xi$ and 
\begin{equation*}
\abs{\partial_x^{\alpha} \partial_{\xi}^{\beta} c(x,\xi)} \leq C_{\alpha \beta} h^{-\sigma\abs{\alpha+\beta}} \br{\xi}^m
\end{equation*}
for all $\alpha,\beta$. If $c \in S^m_{\sigma}$ we define an operator $C = \mOp_h(c)$ by 
\begin{equation*}
Cf(x) = (2\pi)^{-n} \int_{\mR^n} e^{ix\cdot\xi} c(x,h\xi) \hat{f}(\xi) \,d\xi.
\end{equation*}
\end{definition}

Note that we use the standard quantization instead of Weyl quantization in the definition of the operators. We will need the following basic properties.

\begin{prop} \label{psdoproperties}
\cite{dimassisjostrand}, \cite{salo} Let $c \in S^m_{\sigma}$ with $m \in \mR$ and $0 \leq \sigma < 1/2$.
\begin{enumerate}
\item[(a)] 
If $m = 0$ and $\delta \in \mR$ then $\mOp_h(c)$ is bounded $L^2_{\delta} \to L^2_{\delta}$, and there is a constant $M$ with 
\begin{equation*}
\norm{\mOp_h(c)}_{L^2_{\delta} \to L^2_{\delta}} \leq M
\end{equation*}
for $0 < h \leq h_0$.
\item[(b)] 
$hD_{x_j} \mOp_h(c) = \mOp_h(c) hD_{x_j} + h \mOp_h(D_{x_j} c)$.
\item[(c)] 
If $c \in S^m_{\sigma}$ and $d \in S^{m'}_{\sigma}$ then $\mOp_h(c) \mOp_h(d) = \mOp_h(r)$ where $r \in S^{m+m'}_{\sigma}$ satisfies for any $N$ 
\begin{equation*}
r = \sum_{\abs{\alpha} < N} \frac{h^{\abs{\alpha}} \partial_{\xi}^{\alpha} c D_x^{\alpha} d}{\alpha!} + h^{N(1-2\sigma)} S^{m+m'}_{\sigma}.
\end{equation*}
Also, $[\mOp_h(c), \mOp_h(d)] = \mOp_h(s)$ where $s \in S^{m+m'}_{\sigma}$ and 
\begin{equation*}
s = \frac{h}{i} H_{c} d + h^{2(1-2\sigma)} S^{m+m'}_{\sigma}
\end{equation*}
where $H_c = \nabla_{\xi} c \cdot \nabla_x - \nabla_x c \cdot \nabla_{\xi}$ is the Hamilton vector field of $c$.
\end{enumerate}
\end{prop}

Finally, to manage the nonsmooth coefficients, we introduce the standard mollifier $\chi_{\delta}(x) = \delta^{-n} \chi(x/\delta)$ where $\chi \in C_c^{\infty}(\mR^n)$, $0 \leq \chi \leq 1$, and $\int \chi \,dx = 1$. Fix $\sigma_0$ with $0 < \sigma_0 < 1/3$, and consider the $h$-dependent smooth approximation 
\begin{equation*}
A^{\sharp} = A \ast \chi_{\delta},
\end{equation*}
with the specific choice 
\begin{equation*}
\delta = h^{\sigma_0}.
\end{equation*}
We write $A^{\flat} = A - A^{\sharp}$, and note the following standard estimates whose proof is included for completeness.

\begin{lemma} \label{lemma:mollifier_estimates}
If $0 < \varepsilon_0 < \varepsilon$, then as $h \to 0$ 
\begin{equation*}
\norm{\br{x}^{1+\varepsilon} \partial_x^{\alpha} A^{\sharp}}_{L^{\infty}} = O(h^{-\sigma_0 \abs{\alpha}}), \qquad \norm{\br{x}^{1+\varepsilon_0}A^{\flat}}_{L^{\infty}} = o(1).
\end{equation*}
\end{lemma}
\begin{proof}
If $f \in L^1_{\text{loc}}(\mR^n)$ and $r$ is a real number we have 
\begin{equation*}
\br{x}^r \partial_x^{\alpha} (f \ast \chi_{\delta})(x) = \delta^{-\abs{\alpha}} \int K(x,y) \br{y}^r f(y) \,dy
\end{equation*}
where $K(x,y) = \frac{\br{x}^r}{\br{y}^r} \delta^{-n} \partial_x^{\alpha}\chi(\frac{x-y}{\delta})$. If $\delta$ is small enough one sees that $\abs{K(x,y)} \leq 2^{\abs{r}} \delta^{-n} \abs{\partial_x^{\alpha}\chi(\frac{x-y}{\delta})}$, and 
\begin{equation*}
\int \abs{K(x,y)} \,dx \leq C_{r,\alpha}, \quad \int \abs{K(x,y)} \,dy \leq C_{r,\alpha}.
\end{equation*}
Schur's lemma implies $\norm{\br{x}^r \partial_x^{\alpha}(f \ast \chi_{\delta})}_{L^p} \leq C_{r,\alpha} \delta^{-\abs{\alpha}} \norm{\br{x}^r f}_{L^p}$. This and \eqref{w_cgo_assumption} give the $L^{\infty}$ estimates for $A^{\sharp}$.

For the estimates on $A^{\flat}$, we write 
\begin{equation*}
\br{x}^r (f \ast \chi_{\delta} - f)(x) = (g \ast \chi_{\delta} - g)(x) + \int K(x,y) g(y) \,dy
\end{equation*}
where $g = \br{x}^r f$ and $K(x,y) = \chi_{\delta}(x-y)\big[ \frac{\br{x}^r}{\br{y}^r}-1\big]$. Since 
\begin{equation*}
\abs{K(x,y)} \leq C_r \delta \chi_{\delta}(x-y)
\end{equation*}
we have from Schur's lemma that $\norm{\int K(\,\cdot\,,y)g(y) \,dy}_{L^p} \to 0$ as $\delta \to 0$, for $1 \leq p \leq \infty$. For $A^{\flat}$ it is enough to note that $g = \br{x}^{1+\varepsilon_0} A$ is bounded and uniformly continuous, so $g_{\delta} - g \to 0$ in $L^{\infty}$.
\end{proof}

We will use a decomposition $Q = Q^{\sharp} + Q^{\flat}$ where $Q^{\sharp} = 2A^{\sharp} \cdot (hD + \hat{\rho})$ and $Q^{\flat} = 2A^{\flat} \cdot (hD + \hat{\rho})$. Then, we will use pseudodifferential operators to conjugate away the smooth part $Q^{\sharp}$, and when $h$ is small the nonsmooth part $Q^{\flat}$ will be negligible. We are finally ready to give the pseudodifferential conjugation argument.

\begin{prop} \label{prop:conjugation}
Given $\sigma$ with $\sigma_0 < \sigma < 1/3$, there exist $c,\tilde{c},s \in S^0_{\sigma}$ and $\beta > 0$ such that 
\begin{equation} \label{conjugationidentity}
(P + hQ^{\sharp})C = \tilde{C}P + h^{1+ \beta} \br{x}^{-1} S.
\end{equation}
Further, $C$ and $\tilde{C}$ are elliptic, in the sense that $c$ and $\tilde{c}$ are nonvanishing for small $h$.
\end{prop}
\begin{proof}
Suppose $c \in S^0_{\sigma}$ is any symbol. We use Proposition \ref{psdoproperties} and compute 
\begin{multline} \label{firstcommutation}
(P + hQ^{\sharp})C = CP + h \mOp_h(\frac{1}{i} H_p c + 2((\xi + \hat{\rho}) \cdot A^{\sharp}) c) \\
 + h^2 \mOp_h(-\Delta_x c + 2A^{\sharp} \cdot D_x c).
\end{multline}
Here $p(\xi) = \xi^2 + 2\hat{\rho} \cdot \xi$ is the symbol of $P$, and $H_p = 2(\xi + \hat{\rho}) \cdot \nabla_x$ is the Hamilton vector field. The last term is of the form $h^{2-2\sigma} \mOp_h\,S^0_{\sigma}$, and since $\sigma < 1/2$ this has order lower than one. Thus, we would like to choose $c$ such that 
\begin{equation} \label{creq}
\frac{1}{i} H_p c + 2((\xi + \hat{\rho}) \cdot A^{\sharp})c = 0.
\end{equation}
This is a Cauchy-Riemann type equation near the zero set 
\begin{equation*}
p^{-1}(0) = \{ \xi \in \mR^n \,;\, \abs{\xi+\nu_1} = 1, \xi \cdot \nu_2 = 0 \}.
\end{equation*}
Since the principal part $P$ of $P + h Q^{\sharp}$ is elliptic away from the zero set, it will be sufficient to solve \eqref{creq} near $p^{-1}(0)$, and the ellipticity will take care of the rest.

Consider a neighborhood of $p^{-1}(0)$, 
\begin{equation*}
U = U(\delta) = \{ \xi \in \mR^n \,;\, 1-\delta < \abs{\xi+\nu_1} < 1+\delta,\ \abs{\xi \cdot \nu_2} < \delta \}
\end{equation*}
where $\delta = \frac{1}{200}$. We introduce frequency cutoffs $\psi, \psi_1 \in C_c^{\infty}(U(\delta))$ with $\psi_1(\xi) = 1$ on $U(\delta/2)$ and $\psi(\xi) = 1$ near $\supp(\psi_1)$, and also spatial cutoffs $\chi, \chi_1 \in C_c^{\infty}(B(0,1))$ with $\chi_1(x) = 1$ for $\abs{x} \leq 1/2$, and $\chi(x) = 1$ near $\supp(\chi_1)$. The spatial cutoffs will actually be adapted to a ball $B = B(0,M)$, where $M = h^{-\theta}$ is a large parameter depending on $h$, and $\theta = \sigma - \sigma_0 > 0$.

The symbol $c$ is chosen as $c = e^{i\chi_1(h^{\theta} x)\phi}$, where $\phi$ is the solution provided by Lemma \ref{lemma:crparam_solution} below (with $\gamma_1(\xi) = \xi + \nu_1$ and $\gamma_2(\xi) = (1-h^2 \lambda)^{1/2} \nu_2$) to the equation 
\begin{equation} \label{phieq}
(\xi+\hat{\rho}) \cdot \nabla_x \phi(x,\xi) = -\psi(\xi) \chi(h^{\theta} x) (\xi+\hat{\rho}) \cdot A^{\sharp}(x), \quad (x,\xi) \in B \times U.
\end{equation}
The lemma also implies the norm estimates 
\begin{equation*}
\abs{\partial_x^{\alpha} \partial_{\xi}^{\beta} \phi(x,\xi)} \leq C_{\alpha \beta \varepsilon} \Big( \sum_{\abs{\gamma} \leq \abs{\alpha+\beta}} \norm{\br{x}^{1+\varepsilon} \partial_x^{\gamma} A^{\sharp}}_{L^{\infty}} \Big) h^{-\theta\abs{\beta}} \br{x}^{-\varepsilon},
\end{equation*}
and Lemma \ref{lemma:mollifier_estimates} in turn gives 
\begin{equation} \label{phiestimates}
\abs{\partial_x^{\alpha} \partial_{\xi}^{\beta} \phi(x,\xi)} \leq C_{\alpha \beta \varepsilon} h^{-\sigma\abs{\alpha+\beta}} \br{x}^{-\varepsilon},
\end{equation}
for $(x,\xi) \in B \times U$. The fact that $\phi$ vanishes when $\xi$ is outside of $U$ shows that $c \in S^0_{\sigma}$, and $c$ is nonvanishing.

With $c$ as above, the left hand side of \eqref{creq} can be written as 
\begin{equation*}
\frac{1}{i} H_p c + 2((\xi+\hat{\rho}) \cdot A^{\sharp})c = b_1 p + \br{x}^{-1} s_1
\end{equation*}
where 
\begin{align*}
b_1 &= \frac{1-\psi_1(\xi)}{p} (\frac{1}{i} H_p c + 2((\xi+\hat{\rho}) \cdot A^{\sharp})c), \\
s_1 &= \psi_1(\xi) \br{x} (\frac{1}{i} H_p c + 2((\xi+\hat{\rho}) \cdot A^{\sharp})c).
\end{align*}
Since $\frac{1-\psi_1(\xi)}{p} \in S^{-2}_{\sigma}$, we get $b_1 \in h^{-\sigma} S_{\sigma}^{-1}$. For $s_1$ we use \eqref{phieq} to obtain 
\begin{equation*}
s_1 = 2 c \psi_1(\xi) \left[ h^{\theta} \br{x} (\xi + \hat{\rho}) \cdot \nabla \chi_1(h^{\theta} x) \phi + (1-\chi_1(h^{\theta} x)) (\xi + \hat{\rho}) \cdot \br{x} A^{\sharp} \right].
\end{equation*}
This is a sum of two terms where the first term is in $h^{\theta \varepsilon} S^0_{\sigma}$, using that $\norm{h^{\theta} \br{x} \nabla \chi_1(h^{\theta} x)}_{L^{\infty}} < \infty$, the estimates \eqref{phiestimates}, and the fact that $\br{x} \sim h^{-\theta}$ on $\supp(\nabla \chi_1(h^{\theta} \,\cdot\,))$. Also the second term is in $h^{\theta \varepsilon} S^0_{\sigma}$, which follows since $\abs{\br{x} A^{\sharp}(x)} \leq C_{\varepsilon} h^{\varepsilon \theta} \norm{\br{x}^{1+\varepsilon} A^{\sharp}}_{L^{\infty}}$ on $\supp(1-\chi_1(h^{\theta} x))$.

Going back to \eqref{firstcommutation}, we have proved that 
\begin{equation*}
(P + hQ^{\sharp})C = \tilde{C}P + h^{1+ \beta} \br{x}^{-1} S
\end{equation*}
where $\tilde{c} = c + hb_1$ is in $S^0_{\sigma}$ and nonvanishing for small $h$, and we have chosen $\beta = \min\{1-2\sigma-\theta, \theta\varepsilon\} > 0$. One has 
\begin{equation*}
s = h^{-\beta} s_1 + h^{1-\beta}\br{x}(-\Delta_x c + 2A^{\sharp} \cdot D_x c).
\end{equation*}
Then $s \in S^0_{\sigma}$, and the proof is finished.
\end{proof}

The proof of the preceding result is complete modulo Lemma \ref{lemma:crparam_solution} which is deferred to the end of the section. We move to the proof of Proposition \ref{prop:cgosolutions}.

\begin{proof}[Proof of Proposition \ref{prop:cgosolutions}]
We start by showing existence of solutions to \eqref{conjugated_equation}. Using the notation above, we need to solve 
\begin{equation} \label{expanded_eq}
(P_{\rho} + 2A^{\sharp} \cdot D_{\rho} + 2A^{\flat} \cdot D_{\rho} +A^2 + D \cdot A + V)v = f.
\end{equation}
We try a solution of the form $v = G_{\rho} w$ with $w \in L^2_{\delta+1}$. Then $v = C C^{-1} G_{\rho} w$, where $C^{-1}$ is the inverse of $C$ on $L^2_{\delta}$ which exists for small $h$. Inserting this in \eqref{expanded_eq} and using \eqref{conjugationidentity} gives 
\begin{equation*}
(M+T)w = f
\end{equation*}
where 
\begin{eqnarray}
 & M = \tilde{C} P_{\rho} C^{-1} G_{\rho}, & \\
 & T = h^{-1+ \beta} \br{x}^{-1} S C^{-1} G_{\rho} + 2A^{\flat} \cdot D_{\rho} G_{\rho} + (A^2 + D \cdot A + V) G_{\rho}. \label{t_def}
\end{eqnarray}
The mapping properties of $G_{\rho}$ and $S$, together with the estimates in Lemma \ref{lemma:mollifier_estimates} and Sobolev embedding, show that $\norm{T}_{L^2_{\delta+1} \to L^2_{\delta+1}} = o(1)$ as $\abs{\rho} \to \infty$. More precisely, the term involving $V$ (the term with $D \cdot A$ is similar) can be handled by taking $V_0 = \br{x} V$ and by considering the smooth approximation 
\begin{equation*}
V^{\sharp}(x) = \br{x}^{-1} \chi(x/\abs{\rho}) (V_0 \ast \chi_{\alpha} )(x)
\end{equation*}
where $\chi_{\alpha}$ is the mollifier considered above and $\alpha = \abs{\rho}^{-\sigma_0}$. If $V^{\flat} = V - V^{\sharp}$ one obtains 
\begin{eqnarray*}
 & \norm{\br{x} V^{\sharp}}_{L^{\infty}} \leq \norm{V_0}_{L^n} \norm{\chi_{\alpha}}_{L^{n/(n-1)}} \leq C \abs{\rho}^{\sigma_0}, & \\
 & \norm{\br{x} V^{\flat}}_{L^n} \leq \norm{(1-\chi(x/\abs{\rho}))  (V_0 \ast \chi_{\alpha} )}_{L^n} + \norm{V_0 \ast \chi_{\alpha} - V_0}_{L^n} = o(1). & 
\end{eqnarray*}
The embedding $H^1 \subseteq L^{\frac{2n}{n-2}}$ and the estimates for $G_{\rho}$ imply 
\begin{align*}
\norm{V G_{\rho}}_{L^2_{\delta+1} \to L^2_{\delta+1}} \leq \frac{C}{\abs{\rho}} \norm{\br{x} V^{\sharp}}_{L^{\infty}} + C \norm{\br{x} V^{\flat}}_{L^n}
\end{align*}
which gives the required result.

We want to show that for $M$ there is an explicit inverse $N = P_{\rho} C G_{\rho} \tilde{C}^{-1}$. Using \eqref{conjugationidentity}, we can write $M$ and $N$ as 
\begin{eqnarray}
 & M = I + h^{-1} Q^{\sharp} G_{\rho} - h^{-1+ \beta} \br{x}^{-1} S C^{-1} G_{\rho}, & \label{m_alternative} \\
 & N = I - h^{-1} Q^{\sharp} C G_{\rho} \tilde{C}^{-1} + h^{-1+ \beta} \br{x}^{-1} S G_{\rho} \tilde{C}^{-1}.
\end{eqnarray}
It follows that both $M$ and $N$ are bounded on $L^2_{\delta+1}$, with norms uniformly bounded in $\rho$ when $\abs{\rho}$ is large. Now, if $u, f \in L^2_{\delta+1}$ one can show that $Mu = f$ if and only if $u = N f$, by Proposition \ref{prop:grho_estimates} and the boundedness of pseudodifferential operators on weighted Sobolev spaces. This gives that $N = M^{-1}$. Then, for $\rho$ large, $(M+T)^{-1}$ exists by Neumann series and we obtain a solution 
\begin{equation*}
v = G_{\rho} (M+T)^{-1} f.
\end{equation*}
The norm estimates follow from the mapping properties of $G_{\rho}$.

It remains to show uniqueness of solutions to \eqref{conjugated_equation}. Suppose $v \in H^1_{\delta}$ and $(P_{\rho} + 2A \cdot D_{\rho} + \tilde{V}) v = 0$. We can rewrite this as $P_{\rho} v = w$ with $w \in L^2_{\delta+1}$, and Proposition \ref{prop:grho_estimates} implies $v = G_{\rho} w$. We may now write $v = C C^{-1} G_{\rho} w$ and argue as in Step 1, to obtain that $(M+T) w = 0$. The invertibility of $M+T$ was shown above, and consequently $w = 0$ and $v = 0$.
\end{proof}

We may now prove the main result.

\begin{proof}[Proof of Proposition \ref{prop:cgo_solutions_main}]
We go back to the proof of Proposition \ref{prop:cgosolutions} and note that \eqref{m_alternative} and \eqref{t_def} imply 
\begin{equation} \label{ik_mt}
I+K_{\rho} = M+T.
\end{equation}
Since for large $\abs{\rho}$, $M+T$ was bounded and invertible on $L^2_{\delta+1}$ with the norm of $(M+T)^{-1}$ uniformly bounded, the same is true for $I+K_{\rho}$.

The condition that $(H-\lambda)u = 0$ with $u = u_{\rho}$ given by \eqref{urho_def} is equivalent with \eqref{conjugated_equation} where $f = -2A \cdot \rho - \tilde{V}$. By Proposition \ref{prop:cgosolutions}, for $\abs{\rho}$ large there is a unique solution $v_{\rho} \in H^1_{\delta}$ given by 
\begin{equation*}
v_{\rho} = G_{\rho} (M+T)^{-1} f.
\end{equation*}
Then \eqref{vkrho_def} follows from \eqref{ik_mt}.

It remains to prove the asymptotics \eqref{arho_asymptotics}, \eqref{rrho_asymptotics}. We use the approximation scheme in Lemma \ref{lemma:mollifier_estimates}, now choosing $\sigma_0 > 0$ small enough. With $\theta > 0$ also small, from the proof of Proposition \ref{prop:conjugation} we guess that the main term in the asymptotics should be 
\begin{equation*}
a_{\rho}(x) = e^{i \chi_{\rho} \phi^{\sharp}}
\end{equation*}
where $\rho$ is given by \eqref{rhoh_def}, $\chi_{\rho}(x) = \chi(h^{\theta} x)$, and $\phi^{\sharp}$ is given by 
\begin{equation} \label{phisharp_cgo}
\phi^{\sharp}(x) = -\frac{1}{2\pi} \int_{\mR^2} \frac{1}{y_1+i y_2} (\nu_1 + i\nu_2) \cdot A^{\sharp}(x-y_1 \nu_1 - y_2 \nu_2) \,dy_1 \,dy_2,
\end{equation}
and $\chi \in C_c^{\infty}(\mR^n)$ with $0 \leq \chi \leq 1$, $\chi = 1$ for $\abs{x} \leq 1/2$, and $\chi = 0$ for $\abs{x} \geq 1$. Using Lemma \ref{lemma:cr_solution} below, we see that $(\nu_1+i\nu_2) \cdot \nabla \phi^{\sharp} = -(\nu_1+i\nu_2) \cdot A^{\sharp}$, and one has the estimates 
\begin{equation*}
\abs{\partial^{\alpha} \phi^{\sharp}(x)} \leq C_{\alpha,\varepsilon} h^{-\sigma_0 \abs{\alpha}} \br{x_T}^{-\varepsilon} \br{x_{\perp}}^{-1}
\end{equation*}
by Lemma \ref{lemma:mollifier_estimates}. Here $x_T = (x \cdot \nu_1)\nu_1 + (x \cdot \nu_2)\nu_2$ and $x_{\perp} = x-x_T$. This shows \eqref{arho_asymptotics}.

Let $u_{\rho}$ be the solution \eqref{urho_def} with $v_{\rho} \in H^1_{\delta}$, and define $r_{\rho} = 1-a_{\rho}+v_{\rho}$. Since $a_{\rho}-1 \in C^{\infty}_c$ we have $r_{\rho} \in H^1_{\delta}$. Also, $r_{\rho}$ satisfies 
\begin{equation} \label{rrho_equation}
e^{-i\rho \cdot x} (H-\lambda) e^{i\rho \cdot x} r_{\rho} = -f
\end{equation}
with 
\begin{multline} \label{flong}
f = (P_{\rho} + 2A \cdot D_{\rho} + \tilde{V}) a_{\rho} = e^{i\chi_{\rho}\phi^{\sharp}} \Big[ -i\chi_{\rho}\Delta\phi^{\sharp} - 2i \nabla\chi_{\rho} \cdot \nabla\phi^{\sharp} - i\phi^{\sharp}\Delta\chi_{\rho} \\
 + (\chi_{\rho} \nabla\phi^{\sharp} + \phi^{\sharp}\nabla \chi_{\rho})^2 + 2\rho \cdot (\nabla \chi_{\rho}) \phi^{\sharp} + 2\rho \cdot (\nabla \phi^{\sharp}) \chi_{\rho} \\
  + 2A \cdot (\nabla \chi_{\rho})\phi^{\sharp} + 2A \cdot (\nabla \phi^{\sharp}) \chi_{\rho} + 2A^{\sharp} \cdot \rho + 2A^{\flat} \cdot \rho + \tilde{V} \Big].
\end{multline}

We note that when $-1 < \delta < 0$ and $s > 0$ is small, 
\begin{equation} \label{ftermestimate}
\norm{\chi_{\rho} \partial^{\alpha} \phi^{\sharp}}_{L^2_{\delta+1}} \leq C h^{-\theta(1-s)} \norm{\partial^{\alpha} \phi^{\sharp}}_{L^2_{\delta+s}} \leq C_{\alpha} h^{-\sigma_0 \abs{\alpha}-\theta(1-s)}
\end{equation}
by Lemma \ref{lemma:cr_solution} and since $\norm{\br{x} \partial^{\alpha} A^{\sharp}}_{L^2} \leq C_{\alpha} h^{-\sigma_0 \abs{\alpha}}$. This and \eqref{ltwo_cgo_assumption}, or its consequence 
\begin{eqnarray} \label{wsharpltwo_estimate}
 & \norm{\br{x} A^{\sharp}}_{L^2} = O(1), \quad \norm{\br{x} A^{\flat}}_{L^2} = o(1), & 
\end{eqnarray}
can be used to show that the $L^2_{\delta+1}$ norms of most terms in \eqref{flong} are $o(\abs{\rho})$ as $\abs{\rho} \to \infty$. The worst terms in \eqref{flong} cancel because of the equation for $\phi^{\sharp}$ and \eqref{rhoh_def}, or more precisely because 
\begin{multline*}
2\chi_{\rho}(\rho \cdot \nabla \phi^{\sharp}) + 2\rho \cdot A^{\sharp} \\
 = 2i h^{-1} \chi_{\rho}  ((1-h^2 \lambda)^{1/2} - 1) \nu_2 \cdot (\nabla \phi^{\sharp} + A^{\sharp}) + 2(1-\chi_{\rho})(\rho \cdot A^{\sharp})
\end{multline*}
and the $L^2_{\delta+1}$ norm of this is $o(\abs{\rho})$ by \eqref{ftermestimate}, \eqref{wsharpltwo_estimate} and since $\abs{x} \geq C h^{-\theta}$ on $\supp(1-\chi_{\rho})$. Thus $\norm{f}_{L^2_{\delta+1}} = o(\abs{\rho})$, and since $r_{\rho}$ is the unique $H^1_{\delta}$ solution of \eqref{rrho_equation}, Proposition \ref{prop:cgosolutions} shows \eqref{rrho_asymptotics}.
\end{proof}

To end this section, we give the two lemmas which were used in the proofs of Proposition \ref{prop:cgo_solutions_main} and Proposition \ref{prop:conjugation}. In both cases it is straightforward to check that the given function is a solution and satisfies the required estimates (see Lemmas 3.1 and 3.2 in \cite{salo}).

\begin{lemma} \label{lemma:cr_solution}
Let $\gamma_1, \gamma_2 \in \mR^n$ with $\abs{\gamma_j} = 1$ and $\gamma_1 \cdot \gamma_2 = 0$. If $f \in C^{\infty}(\mR^n)$ satisfies $\norm{\br{x}^{1+\varepsilon} \partial^{\alpha} f}_{L^{\infty}} < \infty$ for all $\alpha$, then the equation $(\gamma_1 + i\gamma_2) \cdot \nabla \phi = f$
has a solution $\phi \in C^{\infty}(\mR^n)$, given by 
\begin{equation*}
\phi(x) = \frac{1}{2\pi} \int_{\mR^2} \frac{1}{y_1+iy_2} f(x-y_1 \gamma_1-y_2 \gamma_2) \,dy_1 \,dy_2,
\end{equation*}
which satisfies 
\begin{equation*}
\abs{\partial^{\alpha} \phi(x)} \leq C_{\alpha,\varepsilon} \norm{\br{x}^{1+\varepsilon} \partial^{\alpha} f}_{L^{\infty}} \br{x_T}^{-\varepsilon} \br{x_{\perp}}^{-1},
\end{equation*}
where $x_T = (x \cdot \gamma_1) \gamma_1 + (x \cdot \gamma_2) \gamma_2$ and $x_{\perp} = x - x_T$. Also, if $-1 < \delta < 0$ one has the estimates \cite{sylvesteruhlmann}
\begin{equation*}
\norm{\partial^{\alpha} \phi}_{L^2_{\delta}} \leq C \norm{\partial^{\alpha} f}_{L^2_{\delta+1}}.
\end{equation*}
\end{lemma}

\begin{lemma} \label{lemma:crparam_solution}
Let $U \subseteq \mR^n$ be open, and suppose $\gamma_j(\xi)$ ($j = 1,2$) are smooth functions in $U$ satisfying for any $\xi \in U$ 
\begin{equation*} 
1-\delta < \abs{\gamma_j(\xi)} < 1+\delta, \quad \abs{\gamma_1(\xi) \cdot \gamma_2(\xi)} < \delta, \quad \abs{\partial^{\alpha} \gamma_j(\xi)} \leq 1
\end{equation*}
where $\delta < \frac{1}{100}$ and $\abs{\alpha} \geq 1$. Let $B = B(0,M)$ with $M > 1$ given. Then for any $f \in C^{\infty}_c(B \times U)$, the equation 
$(\gamma_1(\xi) + i\gamma_2(\xi)) \cdot \nabla_x \phi(x,\xi) = f(x,\xi)$
has a solution $\phi \in C^{\infty}(B \times U)$, given by 
\begin{equation*} 
\phi(x,\xi) = \frac{1}{2\pi} \int_{\mR^2} \frac{1}{y_1+i y_2} f(x-y_1 \gamma_1(\xi) - y_2 \gamma_2(\xi),\xi) \,dy_1 \,dy_2,
\end{equation*}
which satisfies 
\begin{equation*}
\abs{\partial_x^{\alpha} \partial_{\xi}^{\beta} \phi(x,\xi)} \leq C_{N,\varepsilon} \Big( \sum_{\abs{\gamma+\delta} \leq N} \norm{\br{x}^{1+\varepsilon} \partial_x^{\gamma} \partial_{\xi}^{\delta} f}_{L^{\infty}(\mR^n \times U)} \Big) M^{\abs{\beta}} \br{x}^{-\varepsilon}
\end{equation*}
when $\abs{\alpha+\beta} \leq N$ and $(x,\xi) \in B \times U$.
\end{lemma}

\section{Analytic dependence} \label{sec:analyticity}

We will now proceed to show that the complex geometrical optics solutions $u_{\rho}$ in \eqref{urho_def} depend analytically on $\rho$ in a certain sense. Given the existence of $u_{\rho}$ for large $\rho \in \mC^n$ with $\rho \cdot \rho = \lambda$, this will follow from analytic Fredholm theory as in \cite{uhlmannvasy}.

We will assume that $A$ and $V$ satisfy \eqref{av_assumptions} for some $\gamma_0 > 0$. Let $\nu$ be a fixed vector in $\mR^n$ with $\abs{\nu} = 1$, and let $\lambda > 0$ be fixed. We write $\rho \in \mC^n$ as $\rho = z\nu + \rho_{\perp}$ where $z \in \mC$ and $\rho_{\perp} \in \mC^n$ with $\rho_{\perp} \cdot \nu = 0$. The vectors $\rho$ will be identified with the pairs $(z,\rho_{\perp})$.   Consider the variety 
\begin{equation*}
\Gamma = \{ \rho \in \mC^n \,;\, \rho = z\nu + \rho_{\perp}, \ \abs{\rho} \geq 1, \ \rho_{\perp} \in \mR^n, \text{ and } \rho \cdot \rho = \lambda \}.
\end{equation*}
Identifying $\{\nu\}^{\perp}$ with $\mR^{n-1}$, we view the error term $v = v_{\rho}$ in \eqref{urho_def} as a function of $z \in \mC \smallsetminus \mR$ and $\rho_{\perp} \in \mR^{n-1}$.

The error term $v_{\rho}$ is explicitly given by \eqref{vkrho_def}, and the result will follow by extending all operators in that identity analytically outside the variety $\rho \cdot \rho = \lambda$. The first step is to do this for $G_{\rho}$. The extension follows from a contour integration argument appearing for instance in \cite{eskinralston} and \cite{melrose}. We use the formulation in \cite{uhlmannvasy}. Here $B(X,Y)$ is the space of bounded operators between Banach spaces $X$ and $Y$.

\begin{prop} \label{prop:grho_analytic}
\cite{uhlmannvasy} Suppose that $\gamma > 0$ and fix $\nu \in S^{n-1}$. Then there exists a neighborhood $U$ of $\mR^{n-1} \smallsetminus \{0\}$ in $\mC^{n-1}$ and an analytic map 
\begin{equation*}
(\mC \smallsetminus \mR) \times U \ni (z,\rho_{\perp}) \mapsto \mathscr{G}_{\rho} \in B(e^{-\gamma\br{x}} L^2, e^{\gamma\br{x}} H^2),
\end{equation*}
such that $\mathscr{G}_{\rho} = G_{\rho}$ when $\rho \in \Gamma$.
\end{prop}

We remark that for $\rho \notin \Gamma$, $\mathscr{G}_{\rho}$ may not coincide with the natural Fourier multiplier definition of $G_{\rho}$, which explains the different notation for the analytic extension. The next step is to consider $K_{\rho} = (2A \cdot D_{\rho} + \tilde{V}) G_{\rho}$.

\begin{lemma}
If $\gamma \leq \gamma_0/2$ there is an analytic map 
\begin{equation*}
(\mC \smallsetminus \mR) \times U \ni (z,\rho_{\perp}) \mapsto \mathscr{K}_{\rho} \in B(e^{-\gamma \br{x}}L^2,e^{-\gamma \br{x}}L^2),
\end{equation*}
with values in compact operators, such that $\mathscr{K}_{\rho} = K_{\rho}$ when $\rho \in \Gamma$.
\end{lemma}
\begin{proof}
Define $\mathscr{K}_{\rho} = (2A \cdot D_{\rho} + \tilde{V}) \mathscr{G}_{\rho}$. Analyticity and boundedness are clear from Proposition \ref{prop:grho_analytic}, and compactness follows from the compact embedding $H^1 \to L^2$.
\end{proof}

One could now consider invertibility of $I + \mathscr{K}_{\rho}$ by using the analytic Fredholm theorem in several complex variables as in \cite{kuchment_survey}. However, for our purposes it is enough to consider vectors $\rho$ parametrized by one complex variable. The main analyticity result is as follows.

\begin{prop} \label{prop:vrho_analytic}
Let $U_0$ be an open connected set in $\mC$, and let $t \mapsto \rho(t)$ be an analytic map from $U_0$ to $(\mC \smallsetminus \mR) \times U$ such that $\rho(U_0) \cap \Gamma$ contains complex vectors whose norms are arbitrarily large. Assume that $\gamma \leq \gamma_0/2$. There exists a discrete subset $\mathcal{E}_0$ of $U_0$, locally given by the zeros of an analytic function, and an analytic map 
\begin{equation*}
U_0 \smallsetminus \mathcal{E}_0 \ni t \mapsto v_{\rho(t)} \in e^{\gamma \br{x}} H^2
\end{equation*}
such that $v_{\rho}$ with $\rho = \rho(t)$ coincides with \eqref{vkrho_def} when $\rho \in \Gamma$ and $\abs{\rho}$ is large.
\end{prop}

\begin{proof}
We know that $t \mapsto \mathscr{K}_{\rho(t)}$ is an analytic family of compact operators on $e^{-\gamma \br{x}} L^2$ for $\gamma \leq \gamma_0/2$, and that $\mathscr{K}_{\rho} \to 0$ in norm when $\rho \in \Gamma$ and $\abs{\rho} \to \infty$ by Proposition \ref{prop:cgo_solutions_main}. Analytic Fredholm theory \cite[Theorem VI.14]{reedsimon} implies that there is a discrete set $\mathcal{E}_0 \subseteq U_0$, which is locally the set of zeros of an analytic function, such that $(I+\mathscr{K}_{\rho(t)})^{-1}$ is an analytic family of bounded operators on $e^{-\gamma \br{x}} L^2$ whenever $t \in U_0 \smallsetminus \mathcal{E}_0$.

For such $t$ we define 
\begin{equation*}
v_{\rho(t)} = \mathscr{G}_{\rho(t)} (I+\mathscr{K}_{\rho(t)})^{-1} (-2A \cdot \rho(t) - \tilde{V}),
\end{equation*}
and the result follows from Proposition \ref{prop:grho_analytic}.
\end{proof}

\section{Uniqueness result} \label{sec:uniqueness}

We assume the hypotheses in Theorem \ref{thm:maintheorem}, and proceed to prove the theorem. The assumption $\Sigma_{\lambda} = \Sigma_{\lambda}'$ and Lemma \ref{lemma:boundarypairing} imply that 
\begin{equation} \label{vanishing_boundary_pairing}
((2A \cdot D + \tilde{V}) u|u') - (u|(2A' \cdot D + \tilde{V}')u') = 0
\end{equation}
for all $u, u' \in e^{\gamma \br{x}} H^1$ such that $(H-\lambda)u = 0$ and $(H'-\lambda)u' = 0$, where $\gamma < \frac{\gamma_0}{2}$.

We would like to use the solutions constructed in Section \ref{sec:cgo_solutions} as $u$ and $u'$. However, these solutions are constructed only for large $\abs{\rho}$ and they may not be in $e^{\gamma \br{x}} H^1$ when $\gamma$ is small. To get around this we will instead use the solutions in Proposition \ref{prop:vrho_analytic} obtained by analyticity.

We make a standard choice of complex vectors given also in \cite{uhlmannvasy}. Fix $\xi \in \mR^n$, and let $\mu, \nu \in \mR^n$ be unit vectors so that $\{\xi,\mu,\nu\}$ is an orthogonal set. We further require that 
\begin{equation} \label{xi_condition}
2\sqrt{\lambda} < \abs{\xi} < 2\sqrt{\lambda + \frac{\gamma_0^2}{4}}.
\end{equation}
For $t > \sqrt{\frac{\abs{\xi}^2}{4} - \lambda}$, define 
\begin{align*}
\rho &= \rho(t) = \frac{\xi}{2} + (t^2+\lambda-\frac{\abs{\xi}^2}{4})^{1/2} \mu + it\nu, \\
\rho' &= \rho'(t) = -\frac{\xi}{2} + (t^2+\lambda-\frac{\abs{\xi}^2}{4})^{1/2} \mu - it\nu.
\end{align*}
Then $\rho \cdot \rho = \rho' \cdot \rho' = \lambda$.

Let $U_0$ be a connected neighborhood of the half line $(\sqrt{\frac{\abs{\xi}^2}{4} - \lambda}, \infty)$ in $\mC$ such that $\re(t^2+\lambda-\frac{\abs{\xi}^2}{4}) > 0$ for $t \in U_0$ and both $\rho(t)$ and $\rho'(t)$ belong to $(\mC \smallsetminus \mR) \times U$ for $t \in U_0$. Here we have used the notations in Section \ref{sec:analyticity} and the principal branch of the square root, so that $\rho(t)$ and $\rho'(t)$ are analytic maps in $U_0$.

Next, we take $u_{\rho}$ and $u_{\rho'}'$ to be solutions of $(H-\lambda)u = 0$ and $(H'-\lambda)u' = 0$, of the form 
\begin{equation} \label{urho_urhoprime_form}
u_{\rho} = e^{i\rho \cdot x}(1+v_{\rho}), \qquad u_{\rho'}' = e^{i\rho' \cdot x}(1+v_{\rho'}')
\end{equation}
where $v_{\rho}$ and $v_{\rho'}'$ are given by Proposition \ref{prop:vrho_analytic} and can be assumed to be in $e^{\gamma' \br{x}} H^2$ for any small $\gamma' > 0$. This works for all $t$ in the set $U_0 \smallsetminus (\mathcal{E} \cup \mathcal{E}')$ where $\mathcal{E}$ and $\mathcal{E}'$ are discrete subsets of $U_0$ which are locally given by the zeros of analytic functions. Also the set $\mathcal{E} \cup \mathcal{E}'$ is discrete in $U_0$ since it is locally given by the zeros of a product of analytic functions.

Now, if $t \in U_0 \smallsetminus (\mathcal{E} \cup \mathcal{E}')$ and additionally $t < \gamma_0/2$, we may insert the solutions $u_{\rho}$ and $u_{\rho'}$ in \eqref{vanishing_boundary_pairing}. This shows that for such $t$ we have $I(t) = 0$, where 
\begin{multline}
I(t) = ((2A \cdot (D+\rho) + \tilde{V})(1+v_{\rho}) | e^{-ix \cdot \xi}(1+v_{\rho'}') ) \\
 - ( e^{ix \cdot \xi}(1+v_{\rho}) | (2A' \cdot (D+\rho') + \tilde{V}')(1+v_{\rho'}') ).
\end{multline}
By Proposition \ref{prop:vrho_analytic}, $I(t)$ is analytic in $U_0 \smallsetminus (\mathcal{E} \cup \mathcal{E}')$. On the other hand we already saw that $I(t) = 0$ in the intersection of this set and $\{ t < \gamma_0/2 \}$. This intersection must contain a small interval. Then by analyticity $I(t) \equiv 0$, and in particular $I(t) = 0$ as $t \to \infty$ on the positive real axis. But in this case $v_{\rho}$ and $v_{\rho'}'$ are given by Proposition \ref{prop:cgo_solutions_main}, and they have the asymptotics \eqref{arho_asymptotics}, \eqref{rrho_asymptotics}.

If $\phi$ and $\phi'$ are given by \eqref{phi_cgo} for $A$ and $A'$, respectively, a computation using the asymptotics shows 
\begin{align}
0 &= \lim_{t \to \infty} \frac{I(t)}{t} = \lim_{t \to \infty} ( (2(A \cdot \frac{\rho}{t})a_{\rho} | e^{-ix \cdot \xi} a_{\rho'}') - (e^{ix \cdot \xi} a_{\rho} | 2(A' \cdot \frac{\rho'}{t}) a_{\rho'}' ) ) \label{nonlinear_fourier_transform} \\
 &= 2(A \cdot (\mu + i\nu) e^{i\phi} | e^{-ix \cdot \xi} e^{i\phi'}) - 2(e^{ix \cdot \xi} e^{i\phi} | A' \cdot (\mu-i\nu) e^{i\phi'}) \notag \\
 &= 2 \int_{\mR^n} e^{ix \cdot \xi} e^{i\Phi} (\mu + i\nu) \cdot (A-A') \,dx \notag
\end{align}
Here $\Phi = \phi - \overline{\phi'}$ is the Cauchy transform 
\begin{equation*}
\Phi(x) = -\frac{1}{2\pi} \int_{\mR^2} \frac{1}{y_1+iy_2} (\mu+i\nu) \cdot (A-A')(x-y_1 \mu - y_2\nu) \,dy_1 \,dy_2.
\end{equation*}
Now \eqref{nonlinear_fourier_transform} says that a certain nonlinear Fourier transform related to $A-A'$ vanishes for $\xi$ satisfying \eqref{xi_condition}, where $\mu$ and $\nu$ are unit vectors and $\{\xi,\mu,\nu\}$ is an orthogonal set. An argument of Eskin and Ralston \cite{eskinralston}, reproduced in Lemma 6.2 of \cite{salo}, shows that 
\begin{equation*}
\int_{\mR^n} e^{ix \cdot \xi} e^{i\Phi} (\mu + i\nu) \cdot (A-A') \,dx = \int_{\mR^n}  e^{ix \cdot \xi} (\mu + i\nu) \cdot (A-A') \,dx.
\end{equation*}
This means that the nonlinear Fourier transform reduces to the usual one. Choosing suitable vectors $\xi$, $\mu$, $\nu$, we see that the Fourier transform of each component of $d(A-A')$ vanishes on the shell \eqref{xi_condition}. Since $d(A-A')$ is exponentially decaying, so the Fourier transform is analytic, we get $dA \equiv dA'$. Thus the magnetic fields coincide.

Finally, we show that $V = V'$. Since $d(A-A') = 0$, Lemma \ref{lemma:gauge_decay} shows that $A-A' = \nabla \alpha$ with $\alpha \in e^{-\gamma_0\br{x}} W^{2,\infty}$. Then by gauge invariance (Lemma \ref{lemma:gaugeinvariance}), the scattering matrices for the coefficients $(A',V')$ and $(A'+\nabla \alpha,V') = (A,V')$ are the same. We may thus assume that $A = A'$ in the argument. Since the scattering matrices for $(A,V)$ and $(A,V')$ at energy $\lambda > 0$ coincide, from Lemma \ref{lemma:boundarypairing} we obtain 
\begin{equation*}
\int_{\mR^n} (V-V') u \overline{u'} \,dx = 0
\end{equation*}
for solutions in $e^{\gamma\br{x}} H^1$ if $\gamma < \gamma_0/2$.

Take $u = u_{\rho}$ and $u' = u_{\rho'}'$ as in \eqref{urho_urhoprime_form}, with $\rho = \rho(t)$ and $\rho' = \rho'(t)$ as earlier. Repeating the argument given above, we may take the limit as $t \to \infty$ and use the asymptotics in Proposition \ref{prop:cgo_solutions_main}. In particular we have $a_{\rho} \overline{a_{\rho'}'} \to 1$ since $A = A'$, and we obtain 
\begin{equation*}
\int_{\mR^n} e^{ix \cdot \xi} (V-V') \,dx = 0
\end{equation*}
for $\xi$ in the frequency shell \eqref{xi_condition}. Again the exponential decay of coefficients implies that the Fourier transform is analytic, and it follows that $V = V'$. This ends the proof of Theorem \ref{thm:maintheorem}.

\addcontentsline{toc}{chapter}{Bibliography}
\providecommand{\bysame}{\leavevmode\hbox to3em{\hrulefill}\thinspace}
\providecommand{\href}[2]{#2}


\begin{thebibliography}{10}

\bibitem{calderon}
A.~P. Calder{\'o}n, \emph{On an inverse boundary value problem}, Seminar on
  Numerical Analysis and its Applications to Continuum Physics, Soc. Brasileira
  de Matem{\'a}tica, R{\'i}o de Janeiro, 1980.

\bibitem{bukhgeim}
A.~L.~Bukhgeim, \emph{{Recovering the potential from Cauchy
data in two dimensions}},  J. Inverse Ill-Posed Probl. \textbf{16} (2008),
19--33.

\bibitem{dimassisjostrand}
M.~Dimassi and J.~Sj{\"o}strand, \emph{Spectral asymptotics in the
  semi-classical limit}, London Mathematical Society Lecture Note Series 268,
  Cambridge University Press, 1999.

\bibitem{eskinralston}
G.~Eskin and J.~Ralston, \emph{{Inverse scattering problem for the
  Schr{\"o}dinger equation with magnetic potential at a fixed energy}}, Comm.
  Math. Phys. \textbf{173} (1995), 199--224.

\bibitem{eskinralston_ima}
G.~Eskin and J.~Ralston, \emph{{Inverse scattering problems for Schr{\"o}dinger operators with
  magnetic and electric potentials}}, Inverse problems in wave propagation, IMA
  Vol. Math. Appl., vol.~90, Springer, New York, 1997, pp.~147--166.

\bibitem{GrinNov}
P.~Grinevich and R.~G.~Novikov, \emph{Transparent potentials at
fixed energy in dimension two. Fixed-energy dispersion relations for the fast decaying potentials}, Comm. Math. Phys. {\bf 174} (1995), 409-446.

\bibitem{H2}
L.~H{\"o}rmander, \emph{The analysis of linear partial differential
  operators}, vol. I-II, Springer-Verlag, Berlin, 1983.

\bibitem{hormander_ucp}
L.~H{\"o}rmander, \emph{Uniqueness theorems for second order elliptic differential
  equations}, Comm. PDE \textbf{8} (1983), 21--64.

\bibitem{isozaki}
H.~Isozaki, \emph{{Inverse scattering theory for Dirac operators}}, Ann. I. H.
  P. Physique Th{\'e}orique \textbf{66} (1997), 237--270.

\bibitem{kenigponcevega}
C.~Kenig, G.~Ponce, and L.~Vega, \emph{{Smoothing effects and local existence theory for the generalized nonlinear Schr\"odinger equations}}, Invent. Math. \textbf{134} (1998), 489--545.

\bibitem{melrose}
R.~B. Melrose, \emph{Geometric scattering theory}, Cambridge University Press,
  1995.
 
\bibitem{nachman}
A.~Nachman, \emph{{Reconstructions from boundary measurements}}, Ann. of Math. \textbf{128} (1988), 531-576.

\bibitem{nakamurasunuhlmann}
G.~Nakamura, Z.~Sun, and G.~Uhlmann, \emph{{Global identifiability for an
  inverse problem for the Schr{\"o}dinger equation in a magnetic field}}, Math.
  Ann. \textbf{303} (1995), 377--388.

\bibitem{nakamurauhlmann}
G.~Nakamura and G.~Uhlmann, \emph{Global uniqueness for an inverse boundary
  problem arising in elasticity}, Invent. Math. \textbf{118} (1994), 457--474.

\bibitem{novikov}
R.~G.~Novikov, \emph{{A multidimensional inverse spectral problem for the equation $-\Delta\psi+(v(x)-E u(x))\psi=0$}}, Funct. Anal. Appl.
\textbf{22} (1988), 263--272.

\bibitem{novikov_exp}
R.~G.~Novikov, \emph{{The inverse scattering problem at fixed energy for the
  three-dimensional Schr\"odinger equation with an exponentially decreasing
  potential}}, Comm. Math. Phys. \textbf{161} (1994), 569--595.

\bibitem{novikovkhenkin}
R.~G.~Novikov and G.~M.~Khenkin, \emph{{The $\bar{\partial}$-equation in the
  multidimensional inverse scattering problem}}, Russ. Math. Surv. \textbf{42}
  (1987), 109--180.

\bibitem{ramm}
A.~G.~Ramm, \emph{{Recovery of the potential from fixed energy scattering data}},
Inverse Problems \textbf{4} (1988), 877-886.

\bibitem{reedsimon}
M.~Reed and B.~Simon, \emph{Methods of modern mathematical physics}, vol.
  I-II, Academic Press, 1975, 1980.

\bibitem{salothesis}
M.~Salo, \emph{{Inverse problems for nonsmooth first order perturbations of the
  Laplacian}}, Ann. Acad. Sci. Fenn. Math. Diss. \textbf{139} (2004).

\bibitem{salo}
M.~Salo, \emph{{Semiclassical pseudodifferential calculus and the
  reconstruction of a magnetic field}}, Comm. PDE \textbf{31} (2006), 1639--1666.

\bibitem{sun}
Z.~Sun, \emph{{An inverse boundary value problem for Schr{\"o}dinger operators
  with vector potentials}}, Trans. Amer. Math. Soc. \textbf{338} (1993), 953--969.

\bibitem{sun_exponential}
Z.~Sun, \emph{{Note on exponentially growing solutions for the Schr\"odinger equations}}, Commun. Appl. Anal. \textbf{9} (2005), 327--335.

\bibitem{sylvesteruhlmann}
J.~Sylvester and G.~Uhlmann, \emph{A global uniqueness theorem for an inverse
  boundary value problem}, Ann. of Math. \textbf{125} (1987), 153--169.

\bibitem{tolmasky}
C.~Tolmasky, \emph{{Exponentially growing solutions for nonsmooth first-order perturbations of the Laplacian}}, SIAM J. Math. Anal. \textbf{29} (1998), 116--133.

\bibitem{uhlmannasterisque}
G.~Uhlmann, \emph{Inverse boundary value problems and applications},
  Ast{\'e}risque (1992), no.~207, 153--211.

\bibitem{uhlmannvasy}
G.~Uhlmann and A.~Vasy, \emph{Fixed energy inverse problem for exponentially
  decreasing potentials}, Methods Appl. Anal. \textbf{9} (2002), 239--248.

\bibitem{weder}
R.~Weder, \emph{{Completeness of averaged scattering solutions}}, Comm.~PDE \textbf{32} (2007), 675-691.

\bibitem{wederyafaev}
R.~Weder and D.~Yafaev, \emph{{On inverse scattering at a fixed energy for potentials with a regular behaviour at infinity}}, Inverse Problems \textbf{21} (2005), 1937-1952.

\bibitem{kuchment_survey}
M.~G. Zaidenberg, S.~G. Krein, P.~A. Kuchment, and A.~A. Pankov, \emph{Banach bundles and linear operators}, Russian Math. Surveys \textbf{30} (1975), 115--175.

\end{thebibliography}
\end{document}